\documentclass[onefignum,onetabnum]{siamart190516}
\usepackage{lipsum}
\usepackage{amsfonts}
\usepackage{graphicx}
\usepackage{epstopdf}
\usepackage{algorithmic}
\ifpdf
  \DeclareGraphicsExtensions{.eps,.pdf,.png,.jpg}
\else
  \DeclareGraphicsExtensions{.eps}
\fi

\newsiamremark{remark}{Remark}
\newsiamremark{hypothesis}{Hypothesis}
\crefname{hypothesis}{Hypothesis}{Hypotheses}
\newsiamthm{claim}{Claim}

\headers{Asymptotics of a nonlocal system}{J. Scott and T. Mengesha}

\title{Asymptotic Analysis of a Coupled System of Nonlocal Equations with Oscillatory Coefficients\thanks{Support from NSF DMS-1910180 and DMS-1615726 is gratefully acknowledged.}}

\author{James M. Scott\thanks{The University of Tennessee, Knoxville, TN 
  (\email{mengesha@utk.edu}, \email{scott@math.utk.edu}).}
\and Tadele Mengesha\footnotemark[2]}

\usepackage{amsopn}

\usepackage{amssymb}
\usepackage{lmodern}
\usepackage{defs}
\usepackage{fonts}
\usepackage{esint}
\usepackage{cite}
\usepackage{enumerate}

\newcommand{\vphi}{\varphi}

\DeclareMathOperator*{\rarrowop}{\longrightarrow}
\DeclareMathOperator*{\supp}{supp}
\DeclareMathOperator*{\dist}{dist}
\DeclareMathOperator*{\divergence}{div}

\newcommand{\intdm}[3]{\displaystyle \int_{#1} #2 \, \mathrm{d}#3}
\newcommand{\iintdm}[5]{\int_{#1} \int_{#2}  #3 \, \mathrm{d}#4 \, \mathrm{d}#5}

\newcommand{\intdmt}[4]{ \int_{#1}^{#2} #3 \, \mathrm{d}#4}

\newcommand{\EquationReference}[2]{\mathrel{\overset{\makebox[0pt]{\mbox{\normalfont\tiny\sffamily #1}}}{#2}}}

\newcommand{\bfeta}{\boldsymbol{\eta}}
\newcommand{\bfs}{\boldsymbol}

\begin{document}

\maketitle

\numberwithin{equation}{section}



\begin{abstract}
In this paper we study the asymptotic behavior of solutions to systems of strongly coupled integral equations with oscillatory coefficients. The system of equations is motivated by a peridynamic model of the deformation of heterogeneous media that additionally accounts for short-range forces. We consider the vanishing nonlocality limit on the same length scale as the heterogeneity and show that the system's effective behavior is characterized by a coupled system of local equations that are elliptic in the sense of Legendre-Hadamard. This effective system is characterized by a fourth-order tensor that shares properties with Cauchy elasticity tensors that appear in the classical equilibrium equations for linearized elasticity.
\end{abstract}

\section{Introduction and statement of main results}
Given $m>0$ and a vector field ${\bf f} \in \big[ L^{2}(\mathbb{R}^{d}) \big]^d$, we study the asymptotic behavior of solutions to the system of equations given by
\begin{equation}\label{lead-equation}
(m\mathbb{I} - \mathbb{L}^{\epsilon}) {\bf u} = {\bf f}\quad \quad \text{ in $\mathbb{R}^{d}$}\,,
\end{equation}
where for each $\epsilon >0,$ the operator $\mathbb{L}^{\epsilon}$ is defined as 
\begin{equation}\label{lead-operator}
\bbL^{\veps} \bu(\bx) := \frac{\lambda \left( \bx, \frac{\bx}{\veps} \right)}{\veps^{d+2}}  \intdm{\bbR^d}{ \frac{{\mu} \left({\bx\over \veps}\right)+ {\mu}\left(\frac{\by}{\veps} \right)}{2} \rho \left( \frac{\bx-\by}{\veps} \right) \frac{(\bx-\by) \otimes (\bx-\by)}{|\bx-\by|^2} \big( \bu(\by) - \bu(\bx) \big)}{\by}
\end{equation}
with $\rho \in L^1(\bbR^d)$ and the functions $\lambda(\bx, {\bf y})$ and ${\mu}({\bf y})$ are bounded and nondegenerate.
We also assume that ${\mu}(\by)$ is periodic while $\lambda(\bx,\by)$ is periodic in the second variable, both with unit period. 

We will show that for a given ${\bf f}\in  \big[ L^{2}(\mathbb{R}^{d}) \big]^d$, the sequence of solutions ${\bf u}^{\epsilon}$ converge strongly in $\big[ L^{2}(\mathbb{R}^{d}) \big]^d$. Moreover, the limiting vector field solves a strongly coupled system of partial differential equations whose coefficients depend on the effective properties of $\lambda$ and ${\mu}$. 

 This work is motivated by the multiscale analysis of the displacement of heterogeneous media in the peridynamic formulation \cite{Silling2000,Silling2007}, a non-local continuum theory for deformable media that incorporates long range interactions of material points via a force field. Parametrized operators of the type $\mathbb{L}^{\epsilon}$ were  first introduced in the work \cite{Alali2012Multiscale} as a means of representing short range forces in the modeling of the deformation of media with heterogeneities. These short range forces are represented at a ``microscopic" scale relative to the heterogeneities, hence the dependence of $\mu$ and $\lambda$ on $\bx/\veps$.
In the absence of the diffusive scaling $\veps^{-2}$ in the operator $\mathbb{L}^{\epsilon}$, the asymptotic properties of solutions of \eqref{lead-equation} is studied in 
\cite{Alali2012Multiscale,du2016multiscale,DuMengeshaMultiscale15} for both stationary and dynamic problems. In this special case, using the method of two scale convergence \cite{Allaire-Twoscale,Nguetseng-twoscale}, it has been shown that solutions to the nonlocal system of equations that use the operator  \eqref{lead-operator} indeed exhibit multiscale behavior and most importantly, their effective or homogenized behavior can be captured by a vector field that is a solution to {\em a (homogenized)  system of nonlocal equations}. It has also been proved that in the presence of the diffusive scaling $\veps^{-2}$ but in the absence of  oscillatory heterogeneity (i.e. $\lambda  = {\mu}=1$) that the integral operator $\mathbb{L}^{\epsilon}$ converges to the Lam\'e-Navier differential operator from classical linearized elasticity and solutions to the peridynamic-type nonlocal system of equations \eqref{lead-equation} also converge to the solution to the corresponding equations of linearized elasticity. This, which is usually referred as nonlocal-to-local convergence, has been demonstrated throughout the literature; see for example \cite{Emmrich-Weckner2007,MengeshaDuElasticity,Du-Zhou2010,Silling2008}. This paper makes an effort to marry these two sets of results and thus attempts to provide a more complete picture of asymptotic regimes in peridynamic formulations. To be precise, we study the asymptotic properties of the operator $\mathbb{L}^{\epsilon}$ and solutions to \eqref{lead-equation} in the presence of both the diffusive scaling and the oscillatory coefficients corresponding to the same length scale.

Periodic and stochastic homogenization of integro-differential operators is currently developing in a variety of directions; see \cite{schwab2010periodic, bonder2017h} for some examples. The asymptotic analysis we present here follows the argument presented in the recent paper \cite{Piatnitski-Zhizhina} which is focused on the homogenization of nonlocal equations that are based on integral operators with convolution-type kernels. The approach is essentially the classical H--convergence of elliptic (differential) operators \cite{Murat-Tartar,Tartar,jikov2012homogenization} 
applied to nonlocal problems with integrable kernels. It turns  out that, although the operator in \eqref{lead-operator} is vector-valued and the resulting system is strongly coupled, the operator shares important features with the scalar-valued nonlocal operator studied in \cite{Piatnitski-Zhizhina}. In fact, following the approach in \cite{Piatnitski-Zhizhina} we will show that the sequence of  operators $\bbL^{\veps}$ will ``converge" to a second-order system of elliptic differential operators in nondivergence form with variable coefficients.
This convergence will be demonstrated via convergence of resolvents, i.e.\ for large enough $m>0$ the operators $(m \bbI - \bbL^{\veps})^{-1}$ converge strongly in $\big[ L^2(\bbR^d) \big]^d$ to the operator $(m\bbI - \bbL^0)^{-1}$, where  
\begin{equation*}
\bbL^0 \bv (\bx) = \mathfrak{C}(\bx) D^{2}{\bf v}(\bx)\,, \qquad  [\mathfrak{C}(\bx) D^{2}{\bf v}(\bx)]_i := \sum_{\ell jk}c_{ijk \ell}(\bx) \frac{\p^2 v^{\ell}}{\p x^j \p x^k}
\end{equation*}
and $\bbI$ denotes the $d \times d$ identity matrix. The limiting system of partial differential operators closely resembles the equilibrium equation for linearized elasticity in the sense that its  fourth-order tensor of coefficients satisfies the Cauchy symmetry relations from linearized elasticity. 

To state the precise statement of the main result, let us fix some notations. We denote the space of $d \times d$ matrices with real coefficients by $\bbM_d(\bbR)$.
If a vector field $\bv = (v^1, v^2, \ldots, v^d) :\bbR^d \to \bbR^d$ has the property that each of its components $v^i$ belongs to a function space $X$, then we denote the associated vector space of vector fields by $\big[ X \big]^d$. For example, Lebesgue spaces of vector fields will be denoted $\big[ L^p(\bbR^d) \big]^d$. Function spaces of higher-order tensor fields $\mathfrak{V} = (V^{i_1 i_2 \ldots i_n}) : \bbR^d \to \bbR^{d^n}$ will be denoted similarly, i.e.\ the Lebesgue space is written as $\big[ L^p(\bbR^d) \big]^{d^n}$. We denote the space of Schwarz functions by $\cS(\bbR^d)$.  We denote the space of bounded linear operators from a Banach space $Y_1$ to a Banach space $Y_2$ by $\cL(Y_1,Y_2)$.
We designate the set of infinitesimal rigid displacements $\cM$ as
\begin{equation*}
\cM := \left\lbrace \bv(\bx) \, : \, \bv(\bx) =\boldsymbol{\rmQ}\bx + \bm\,, \quad \boldsymbol{\rmQ} \in \bbM_d(\bbR)\,, \quad \boldsymbol{\rmQ}^{\intercal} = - \boldsymbol{\rmQ}\,, \quad \bm \in \bbR^d \right\rbrace\,.
\end{equation*}
The kernel $\rho(\bz)$, in addition to being in $L^1(\bbR^d)$ also satisfies 
\begin{equation}\label{eq:A1}
\begin{split}
\rho(\bz) \geq 0\,, &\qquad \rho(-\bz)=\rho(\bz)\,, \\
	\intdm{\bbR^d}{ \rho(\bz)}{\bz} := a_1 > 0\,, &\qquad \intdm{\bbR^d}{|\bz|^2 \rho(\bz)}{\bz} := a_2 < \infty\,.
\end{split} \tag{A1}
\end{equation}
We also assume that $\rho$ is nondegenerate in the sense that there exists  $\delta_0 > 0$ and a symmetric cone $\Lambda$  with vertex at the origin such that
\begin{equation}\label{eq:A2}
 \Lambda \cap B_{\delta_0}({\bf 0}) \subset \supp \rho\,. 
\tag{A2}
\end{equation}
By ``symmetric cone with vertex at the origin" we mean that there exists an open subset $\cJ$ of the unit sphere  $\bbS^{d-1}$ with  $\cH^{d-1}(\cJ) > 0$ such that the set $\Lambda$ can be written as 
$\Lambda = \left\lbrace \bx \in \bbR^d \, : \, \frac{\bx}{|\bx|} \in \cJ \cup -\cJ \right\rbrace\,.$
The function ${\mu}(\by)$ is assumed to be periodic in $\bx$ and $\by$ with unit period   $\bbT^d := [0,1]^d$ while for each $\bx$, $\lambda(\bx,\cdot)$ is assumed to be periodic in the second variable. Moreover, 
\begin{equation}\label{eq:A3}
\begin{cases}
&0 < \alpha_1 \leq \lambda \left( \bx, \by \right)\,, {\mu}(\by) \leq \alpha_2 < \infty\,, \qquad \bx, \by \in \bbR^d\,, \\
&\forall \by\in \mathbb{R}^{d}, \lambda(\cdot, \by) \in C^{\infty}(\mathbb{R}^{d}). 
\end{cases} \tag{A3}
\end{equation}
Throughout this work we will identify periodic functions defined on all of $\bbR^d$ with functions defined on the torus $\bbT^d = [0,1]^d$. A consequence of the periodicity, such functions satisfy the identity $\intdm{\bbT^d}{g(\by+\bx)}{\bx} = \intdm{\bbT^d}{g(\bx)}{\bx}$, for every $\by \in \bbR^d$\,.  For any $1\leq p<\infty$, the $L^{p}$ space of periodic functions is defined as 
\begin{equation*}
L^{p}(\bbT^{d}) = \{g: \mathbb{R}^{d}\to \bbR : \text{g periodic and } \int_{\bbT^{d}}|g|^{p} \, \rmd \bx <\infty  \}\,. 
\end{equation*}
For a given positive integer $k$, the Sobolev space $\big[ H^{k}(\bbR^d) \big]^d$ consists of functions whose $j^{th}$-order tensor of partial derivatives belongs to $\big[ L^2(\bbR^d) \big]^{d^{j+1}}$ for all $j \leq k$; precisely,
$$
\big[ H^{k}(\bbR^d) \big]^d := \left\lbrace \bu \in \big[ L^2(\bbR^d) \big]^d \, \Big| \, |D^j \bu|   \in  L^2(\bbR^d)\,, \, ) \leq j \leq k \right\rbrace\,,
$$
with norm
$$
\Vnorm{\bu}_{H^{k}(\bbR^d)} :=  \sum_{0 \leq j \leq k} \Vnorm{D^j \bu}_{L^2(\bbR^d)}\,.
$$
It is standard that ${\bu}\in \big[ H^{k}(\bbR^d) \big]^d$ if and only if ${\bu }$ and all of its partial derivatives of up to $k^{th}$ order are all in $\big[ L^{2}(\mathbb{R}^{d}) \big]^d$.  Given a vector field $\bu$, $D{\bf u}$ denotes  the gradient matrix of ${\bf u}$ given by $(D{\bf u})_{ij}:={\p u_{i} \over \p{x_j}} $, and $D^{2}{\bf u}$ is the third-order tensor of second partial derivatives of ${\bf u}$ given by $((D^{2}{\bf u})_{ijk}:={\p^2 u^{i}\over \p{x_k}\p{x_j} } )$. Higher order tensors of partial derivatives will also be denoted in a similar fashion. Finally, we use the standard ``contraction of indices" convention and Einstein summation notation when denoting actions of tensors unless specified otherwise.
We can now state the main result of the paper.
\begin{theorem}\label{thm:MainThm}
Suppose that $\rho, {\mu_s}$ and $\lambda$ satisfy \eqref{eq:A1}, \eqref{eq:A2}, and \eqref{eq:A3}. Then there exists a constant $m_0 >0$ such that for all $m\geq m_0$,  the resolvents $(m\bbI - \bbL^{\veps})^{-1}$ converge strongly to $(m\bbI - \bbL^0)^{-1}$ as $\veps \to 0$. Precisely, for every $\bff \in \big[ L^2(\bbR^d) \big]^d$, if $\bu^{\veps}$ is a solution to \eqref{lead-equation}, then 
\begin{equation*}
\Vnorm{\bu^{\veps} - \bu^0}_{L^2(\bbR^d)} \to 0 \quad \text{ as } \quad \veps \to 0\,,
\end{equation*}
where 
$\bu^0$ solves the equation  
\begin{equation}\label{eq:u0Def}
\left(m\mathbb{I} - \mathbb{L}^{0}\right){\bf u} = {\bf f} \quad\text{in $\mathbb{R}^{d}$}\,,
\end{equation}
and the operator $\mathbb{L}^{0}:\big[ H^2(\bbR^d) \big]^d \to \big[ L^2(\bbR^d) \big]^d$ is a second-order system of linear differential operators
$
\mathbb{L}^{0}{\bf u} = \mathfrak{C}(\bx) D^{2}{\bf u}
$ 
whose tensor of coefficients $\mathfrak{C}(\bx)$ is elastic and is infinitely differentiable. 
\end{theorem}
Some remarks are in order.  In the theorem,  by {\em elastic tensor} we mean a fourth-order tensor $\mathfrak{C}(\bx)$  satisfying the symmetries
\begin{equation}\label{eq:ElasticityTensor1}
c^{ijk \ell}(\bx) = c^{k \ell ij}(\bx)\,, \qquad c^{ijk \ell}(\bx) = c^{ji k \ell}(\bx) = c^{ij \ell k}(\bx)
\end{equation}
for every $\bx \in \bbR^d$, and that for some positive constants $\gamma_1$ and $\gamma_2$ the inequalities
\begin{equation}\label{eq:ElasticityTensor2}
\begin{split}
\gamma_1 |\boldsymbol{\bbW}|^2 \leq & \Vint{\mathfrak{C}(\bx)\boldsymbol{\bbW}, \boldsymbol{\bbW}} = c^{ijk \ell}(\bx) w_{kl} w_{ij}\,, \\
& \Vint{\mathfrak{C}(\bx)\boldsymbol{\bbW}, \bbV} \leq \gamma_2 |\bbW| |\bbV| 
\end{split}
\end{equation}
hold uniformly in $\bx$ and for all symmetric matrices $\boldsymbol{\bbW} = (w_{ij})$ and $\bbV = (v_{ij})$. We have used the generic inner product notation $\langle\cdot, \cdot\rangle$. The tensor of coefficients $\mathfrak{C}$ of $\bbL^{0}$ will be defined in terms of tensor-valued corrector functions that solve a system of auxiliary cell problems. More importantly, for each $\bx\in\bbR^{d}, $
$\mathfrak{C}(\bx)$ is of the form
\begin{equation}\label{eq:ElasticityTensorForm}
\mathfrak{C}(\bx) =  \left(\int_{\bbT^{d}} {1\over \lambda(\bx, \by)} \, \rmd \by\right)^{-1} \widetilde{\mathcal{C}}, \quad \quad \text{where $\widetilde{\mathcal{C}}$ is a constant tensor.}
\end{equation}
As a consequence of this and the smoothness assumption \eqref{eq:A3} on $\lambda$, $\mathfrak{C}$ is infinitely differentiable. 

(In)equalities \eqref{eq:ElasticityTensor1} and \eqref{eq:ElasticityTensor2} are instrumental in showing that the resolvent $(m\bbI - \bbL^{0})^{-1}$ is well-defined. In fact, \eqref{eq:ElasticityTensor1} and \eqref{eq:ElasticityTensor2} imply that $\mathfrak{C}$ satisfies the Legendre-Hadamard
condition, namely $c^{ijk \ell}\in L^{\infty}(\mathbb{R}^{d})$ and 
\begin{equation*}
\langle\mathfrak{C}(\bx) {\bfs \xi}\otimes {\bfs \eta}, {\bfs \xi}\otimes {\bfs \eta}\rangle=c^{ijk \ell}(\bx)\xi_i\xi_k\eta_j\eta_\ell \geq \gamma_1|{\bfs \xi}|^{2}|{\bfs \eta}|^2,\quad\forall {\bfs \xi}=(\xi_i), \,{\bfs \eta} = (\eta_i)\in \bbR^{d}. 
\end{equation*}
This condition is sufficient to guarantee the existence of the resolvent $(m\bbI - \bbL^0)^{-1}$ using \textit{a priori} estimates for elliptic systems combined with the method of continuity \cite{DoKi09, krylov1996lectures, GiaquintaBook}. In fact, we can apply \cite[Theorem 2.6]{DoKi09} to conclude that there exists $m_0>0$ such that for any $\bff\in [L^{2}(\bbR^{d})]^{d}$ and any $m\geq m_0$, there exists a unique vector field ${\bf u}^{0}=(m\bbI - \bbL^0)^{-1}{\bff} \in [H^{2}(\bbR^{d})]^{d}$ with the estimate 
 \begin{equation*}
 \|{\bf u}^{0}\|_{H^{2}(\bbR^d} \leq C \|{\bf f}\|_{L^{2}(\bbR^d)}, 
 \end{equation*}
proving that the resolvent $(m\bbI - \bbL^0)^{-1}:  \big[ L^2(\bbR^d) \big]^d\to \big[ H^2(\bbR^d) \big]^d $ is a well-defined bounded operator. 
From the smoothness of $\mathfrak{C}$ and standard regularity theory we have that if ${\bff}\in C^{\infty}$, then so is $(m\bbI - \bbL^0)^{-1}{\bff}$, see \cite{DoKi09, krylov1996lectures, GiaquintaBook}.

Finally, if we take ${\mu_s} \equiv 1$ and take $\rho$ to be radial, i.e.\ $\rho(\bz) := \widetilde{\rho}(|\bz|)$, then the constant tensor $\widetilde{\cC} = (\widetilde{c}^{ijk \ell})$ defined in \eqref{eq:ElasticityTensorForm} is exactly the elasticity tensor associated to the Lam\'e-Navier system. 
Specifically, we have 
\begin{equation}\label{eq:LameTensor}
\widetilde{c}^{ijk \ell} = \frac{a_2}{2d(d+2)} \big( \delta_{ij} \delta_{k \ell} + \delta_{i k } \delta_{j \ell} + \delta_{i \ell} \delta_{j k} \big)\,.
\end{equation}
%Note that the Lam\'e parameters in the Lam\'e system specified by \eqref{eq:LameTensor} are both equal to $\frac{a_2}{2d(d+2)}$.
Above, $\delta_{ij}$ denotes the Kronecker $\delta$-function. It is straightforward to check that this constant tensor when used in the definition of $\bbL^0$ gives
\begin{equation}\label{eq:LameSystem}
\bbL^0 \bu(\bx) = {\mu}^0(\bx) \Delta \bu(\bx) + 2 {\mu}^0(\bx) \grad (\divergence \bu(\bx))\,,
\end{equation}
where the Lam\'e parameter ${\mu}^0(\bx)$ is defined as
\begin{equation*}
{\mu}^0(\bx) := \frac{a_2}{2d(d+2)} \left(\int_{\bbT^{d}} {1\over \lambda(\bx, \by)} \, \rmd \by\right)^{-1}\,.
\end{equation*}
This form of the limiting operator is expected; once the heterogeneity ${\mu_s}$ is removed from \eqref{lead-equation} the convergence then resembles the nonlocal-to-local limits considered in \cite{Emmrich-Weckner2007,MengeshaDuElasticity,Du-Zhou2010,Silling2008}. The computation of \eqref{eq:LameTensor} is summarized in the appendix. 

The paper is organized as follows: In Section \ref{tools} we collect tools and general results needed for subsequent sections. In Section \ref{asy} we set up the program of asymptotic analysis via correctors and prove Theorem \ref{thm:MainThm}, deferring proofs regarding existence of the correctors (such as the solvability of the auxiliary cell problem) to Section \ref{aux}.

\section{Tools and preliminaries}\label{tools}
In this section we collect tools that we will need throughout the paper. We will also prove some preliminary results related to the main operators of interest. Other general results that we need later in the paper will also be discussed. 
\subsection{Existence and uniform estimates for resolvents}
For $\rho$ satisfying the condition \eqref{eq:A1} and $\veps > 0$,  define $\rho_{\veps}(\bz) := {1\over \veps^{d}} \rho\left({\bz\over \veps}\right)$. Then $\rho_{\veps} \in L^1(\bbR^d)$ and $\int_{\mathbb{R}^{d}} \rho_{\veps}(\bz) \, \rmd {\bz}  = a_1$. We also denote the symmetrized form of ${\mu} $ by 
\begin{equation*}{\mu_s} (\bx, \by) = {1\over 2}({\mu}(\bx) + {\mu}(\by)), \quad \text{for $\bx, \by\in \bbT^{d}$}.\end{equation*}  The function ${\mu_s}$ is clearly periodic in both variables. We use this notational convention to introduce the matrix-valued functions 
\begin{equation}\label{eq:MatricesInOperator}
 \boldsymbol{\rmK}_{\veps}(\bx) := \rho_{\veps}(\bx) \left( \frac{\bx \otimes \bx}{|\bx|^2 }\right),\quad  \text{and, } 
 \boldsymbol{\rmG}_{\veps}(\bx) := \intdm{\bbR^d}{\boldsymbol{\rmK}_{\veps}(\bx-\by) {\mu_s} \left( \frac{\bx}{\veps}, \frac{\by}{\veps} \right)} {\by}\,, \quad \bx \in \bbR^d\,.
\end{equation}
Notice that for each $\bx,$ and any $\epsilon>0$, $ \boldsymbol{\rmK}_{\veps}(\bx)$ is positive semi-definite matrix and for any vector $\bv$, we have $\langle \boldsymbol{\rmK}_{\veps}(\bx){\bv}, {\bv}\rangle = \rho_{\veps}(\bx)\left|{{\bx} \over |\bx|}\cdot \bv \right|^{2} \geq 0$.  Moreover, by a change of variables, it is clear that $\boldsymbol{\rmK}_{\veps} \in \big[L^1(\bbR^d) \big]^{d \times d}$. The estimate \eqref{eq:A3} on $\mu$ gives $|\boldsymbol{\rmG}_{\veps}(\bx)| \leq \alpha_2 \Vnorm{\rho}_{L^1(\bbR^d)}$ for every $\bx \in \bbR^d$. 

Using these matrix-valued functions, we define the operators $\bbK_{\veps}$ and $\bbG_{\veps}$ for $\bv \in \big[L^2(\bbR^d) \big]^d$   by
\begin{equation*}
\begin{split}
\bbK_{\veps} \bv(\bx) := \intdm{\bbR^d}{ \boldsymbol{\rmK}_{\veps}(\bx-\by) {\mu_s}\big({\bx\over \veps}, {\by\over \veps}\big) \bv(\by)}{\by} \quad  \text{ and }  \quad \bbG_{\veps} \bv (\bx) :=   \boldsymbol{\rmG}_{\veps}(\bx) \bv(\bx)\,. 
\end{split}
\end{equation*}
For $\veps=1$, we simply write  $\boldsymbol{\rmK}(\bx)$, ${\bbK}$,  $\boldsymbol{\rmG}(\bx)$  and $\bbG$ as  opposed to $ \boldsymbol{\rmK}_{1}(\bx)$, ${\bbK}_{1}$, $ \boldsymbol{\rmG}_{1}(\bx)$, and ${\bbG}_{1}$,  respectively. 

We can then rewrite the main operator $\mathbb{L}_{\epsilon}$ introduced in \eqref{lead-operator} as 
\begin{equation*}
\begin{split}
\bbL^{\veps} \bv (\bx) &=  \frac{1}{\veps^2} \lambda \left( \bx, \frac{\bx}{\veps} \right) \Bigg[\intdm{\bbR^d}{ \boldsymbol{\rmK}_{\veps}(\bx-\by){\mu_s} \left( \frac{\bx}{\veps}, \frac{\by}{\veps} \right)  \bv(\by) }{\by} \\
&\qquad \qquad \qquad \qquad - \left(\intdm{\bbR^d}{ \boldsymbol{\rmK}_{\veps}(\bx-\by) {\mu_s} \left( \frac{\bx}{\veps}, \frac{\by}{\veps} \right) }{\by}\right) \bv(\bx)\Bigg] \\
&= \veps^{-2} \lambda \left( \bx, \frac{\bx}{\veps} \right) \left( \bbK_{\veps} {\bf v}(\bx) - \bbG_{\veps} {\bf v}(\bx)\right),
\end{split}
\end{equation*}
The operator $\bbK_{\veps}$ is a combination of convolution and multiplication operators, and can be rewritten as 
\begin{equation*}
\bbK_{\veps} {\bf v}(\bx)= \frac{1}{2}\big( \boldsymbol{\rmK}_{\veps} \ast ( \mathrm{M}_{\mu} \bv ) \big)(\bx) + \frac{1}{2} \big( \rmM_{\mu}  ( \boldsymbol{\rmK}_{\veps} \ast \bv ) \big) (\bx)
\end{equation*}
where $\mathrm{M}_{f}$ denotes the multiplication operator, i.e. $(\mathrm{M}_{f} \bv)(\bx) = f(\bx) \bv(\bx)$ for any bounded function $f$ and the convolution is defined in terms of matrix multiplication, i.e.
\begin{equation*}
\big[ \boldsymbol{\rmK} \ast \bv(\bx) \big]_i =  \int_{\bbR^d}  \boldsymbol{\rmK}^{ij}(\bx-\by) v_j(\by) \, \mathrm{d}\by\,.
\end{equation*}
We will use this formulation of the operator and the following result repeatedly throughout this work. The proof of the following lemma mimics that given in \cite{Piatnitski-Zhizhina} appropriately modified to fit our framework.

\begin{lemma}\label{exist-resolvent}
Assume that $\rho, {\mu}$ and $\lambda$ satisfy \eqref{eq:A1}, \eqref{eq:A2} and \eqref{eq:A3}. 
For each $\veps > 0$, the linear operator $\bbL^{\veps} : \big[ L^2(\bbR^d) \big]^d \to \big[ L^2(\bbR^d) \big]^d$ is  bounded. Moreover, for each $m > 0$, the operator $ m\mathbb{I} - \bbL^{\veps}: \big[ L^2(\bbR^d) \big]^d \to \big[ L^2(\bbR^d)]^{d}$ is an isomorphism with bounded inverse such that there exists a constant  $C$, independent of $\veps$,  with the property that  for any $\bff \in \big[ L^2(\bbR^d) \big]^d$, and any $\veps > 0$
\begin{equation}\label{eq:AprioriEst}
\| (m\mathbb{I} - \bbL^{\veps})^{-1}{\bf f}\|_{L^{2}} \leq C \|{\bf f}\|_{L^{2}}. 
\end{equation}
\end{lemma}
\begin{proof} The boundedness of $\bbL^{\veps}$ follows from a trivial application of Young's inequality as 
\begin{equation*}
\begin{split}
\Vnorm{\bbL^{\veps} \bv}_{L^2(\bbR^d)} &\leq \frac{\alpha_2}{\veps^2} \left(\Vnorm{ \bbK_{\veps} \bv }_{L^2(\bbR^d)} + \Vnorm{\boldsymbol{\rmG}_{\veps} \bv}_{L^2(\bbR^d)}\right) \\
	&\leq \left( \frac{\alpha_2}{\veps} \right)^2 \left(\Vnorm{\rho}_{L^1(\bbR^d)} \Vnorm{\bv}_{L^2(\bbR^d)} + \Vnorm{\rho}_{L^1(\bbR^d)} \Vnorm{\bv}_{L^2(\bbR^d)}\,\right).
\end{split}
\end{equation*}
and so
\begin{equation*}
\Vnorm{\bbL^{\veps}}_{\cL([ L^2(\bbR^d) ]^d \, , \, [L^2(\bbR^d) ]^d)} \leq \frac{2 \alpha_2^2 \Vnorm{\rho}_{L^1(\bbR^d)}}{\veps^2}\,.
\end{equation*}
To prove the invertibility of $ m \bbI - \bbL^{\veps}$ we introduce the weighted Lebesgue space $\big[ L^2_{\nu_{\veps}}(\bbR^d) \big]^d$ with weight $\displaystyle \nu_{\veps}(\bx) = \frac{1}{\lambda \left( \bx, \frac{\bx}{\veps} \right)}$. Then
\begin{equation*}
0 < \frac{1}{\alpha_2} \leq \nu_{\veps}(\bx) \leq \frac{1}{\alpha_1} < \infty\,,
\end{equation*}
which implies $\big[ L^2_{\nu_{\veps}}(\bbR^d) \big]^d = \big[ L^2(\bbR^d) \big]^d$. Thus $\bbL^{\veps} : \big[ L^2_{\nu_{\veps}}(\bbR^d) \big]^d \to \big[ L^2_{\nu_{\veps}}(\bbR^d) \big]^d$ is a bounded linear operator.
Further, $\bbL^{\veps}$ is self-adjoint in $\big[ L^2_{\nu_{\veps}}(\bbR^d) \big]^d$, and $\Vint{\bbL^{\veps}\bu,\bu}_{L^2_{\nu_{\veps}}(\bbR^d)} \leq 0$.
Indeed, splitting the double integral, interchanging the role of $\bx$ and $\by$ and using Fubini's theorem and the fact that $ {\mu_s} \left( \bx, \by \right) = {\mu_s} \left( \by, \bx \right)$,
\begin{equation*}
\begin{split}
\big\langle & \bbL^{\veps}\bu,\bv \big\rangle_{L^2_{\nu_{\veps}}(\bbR^d)} \\
 &= {1\over \epsilon^{2}}\iintdm{\bbR^d}{\bbR^d}{ {\mu_s} \left( \frac{\bx}{\veps}, \frac{\by}{\veps} \right) \Vint{ {\bfs \rmK}_{\epsilon} (\bx-\by) \big( \bu(\by) - \bu(\bx) \big), \bv(\bx)}}{\by}{\bx} \\
	&= \frac{1}{2\epsilon^{2}} \iintdm{\bbR^d}{\bbR^d}{\ldots}{\by}{\bx} + \frac{1}{2\epsilon^{2}} \iintdm{\bbR^d}{\bbR^d}{\ldots}{\by}{\bx} \\
	&=  \frac{1}{2\epsilon^{2}}  \iintdm{\bbR^d}{\bbR^d}{{\mu_s} \left( \frac{\bx}{\veps}, \frac{\by}{\veps} \right) \Vint{ {\bfs \rmK}_{\epsilon} (\bx-\by) \big( \bu(\by) - \bu(\bx) \big), \bv(\bx)}}{\by}{\bx} \\
	&\quad+ \frac{1}{2\epsilon^{2}}  \iintdm{\bbR^d}{\bbR^d}{{\mu_s} \left( \frac{\by}{\veps},  \frac{\bx}{\veps} \right) \Vint{  {\bfs \rmK}_{\epsilon} (\bx-\by)\big( \bu(\bx) - \bu(\by) \big), \bv(\by)}}{\bx}{\by} \\
	&= -  \frac{1}{2\epsilon^{2}}  \iintdm{\bbR^d}{\bbR^d}{{\mu_s} \left( \frac{\bx}{\veps},\frac{\by}{\veps} \right) \Vint{{\bfs \rmK}_{\epsilon} (\bx-\by) \big( \bu(\by) - \bu(\bx) \big), \bv(\by)-\bv(\bx)}}{\by}{\bx}\,.
\end{split}
\end{equation*}
Setting $\bv = \bu$ in the last line gives $\Vint{\bbL^{\veps}\bu,\bu}_{L^2_{\nu_{\veps}}(\bbR^d)} \leq 0$. Repeating these steps in reverse order with the roles of $\bu$ and $\bv$ interchanged, one sees that $\Vint{\bbL^{\veps}\bu,\bv}_{L^2_{\nu_{\veps}}(\bbR^d)} = \Vint{\bu,\bbL^{\veps} \bv}_{L^2_{\nu_{\veps}}(\bbR^d)}$.
Thus for $m>0$ we have $\Vint{(m-\bbL^{\veps}) \bu,\bu}_{L^2_{\nu_{\veps}}(\bbR^d)} \geq m \Vnorm{\bu}_{L^2_{\nu_{\veps}}(\bbR^d)}^2$. This implies by the Cauchy-Schwarz inequality that 
\begin{equation*}
m \Vnorm{\bu}_{L^2_{\nu_{\veps}}(\bbR^d)}^2  \leq \Vint{(m-\bbL^{\veps}) \bu,\bu}_{L^2_{\nu_{\veps}}(\bbR^d)} \leq \Vnorm{(m-\bbL^{\veps}) \bu}_{L^2_{\nu_{\veps}}(\bbR^d)} \Vnorm{\bu}_{L^2_{\nu_{\veps}}(\bbR^d)}\,,
\end{equation*}
and therefore we have 
\begin{equation}\label{eq:L2LowerBoundOnLeps}
m \Vnorm{\bu}_{L^2_{\nu_{\veps}}(\bbR^d)} \leq \Vnorm{(m-\bbL^{\veps}) \bu}_{L^2_{\nu_{\veps}}(\bbR^d)}\,. 
\end{equation}
As a consequence, Range$(m \bbI - \bbL^{\veps})$ is closed and  kernel$(m \bbI - \bbL^{\veps}) = \{ {\bf 0 } \}$. It then follows that Range$(m \bbI - \bbL^{\veps}) = $kernel$(m \bbI - \bbL^{\veps})^{\perp} = \big[ L^2_{\nu_{\veps}}(\bbR^d) \big]^d$. 
Thus $(m \bbI - L^{\veps})^{-1} : \big[ L^2_{\nu_{\veps}}(\bbR^d) \big]^d \to \big[ L^2_{\nu_{\veps}}(\bbR^d) \big]^d$  is a bounded linear operator.
Setting $\bu = (m\bbI - \bbL^{\veps})^{-1}\bv$ for $\bv \in \big[ L^2_{\nu_{\veps}}(\bbR^d) \big]^d$ in \eqref{eq:L2LowerBoundOnLeps} gives a bound of $\frac{1}{m}$ for the operator norm of 
$(m\bbI - \bbL^{\veps})^{-1}$ over $\big[ L^2_{\nu_{\veps}}(\bbR^d) \big]^d$. Notice  that the bound is uniform in $\veps$. Finally, since $\big[ L^2_{\nu_{\veps}}(\bbR^d) \big]^d = \big[ L^2(\bbR^d) \big]^d$, with norms comparable with constants independent of $\veps$ the operator $(m\bbI - \bbL^{\veps})^{-1} : \big[ L^2(\bbR^d) \big]^d \to \big[ L^2(\bbR^d) \big]^d$ is well-defined, linear and bounded, with operator norm bounded independent of $\veps$.
\end{proof}

\subsection{Some operators on function spaces of periodic functions}
As we will see in the next section, the operators $\bbK$ and $\bbG$ will be applied not only to functions in $\big[ L^2(\bbR^d)\big]^d$ but also to functions in $\big[ L^2(\bbT^d)\big]^d $. We now summarize basic properties of these operators as linear maps on $\big[ L^2(\bbT^d)\big]^d$. We begin with the following proposition. 

\begin{proposition}\label{PD-G} Assume that $\rho, {\mu}$ and $\lambda$ satisfy \eqref{eq:A1}, \eqref{eq:A2} and \eqref{eq:A3}. Then 
$\bbG : \big[ L^2(\bbT^d) \big]^d \to \big[ L^2(\bbT^d) \big]^d$ is a bounded invertible linear operator.
\end{proposition}

\begin{proof}
Since the matrix $|\boldsymbol{\rmG}(\bq)| \leq \alpha_2 a_1$ uniformly in $\bq$, $\bbG : \big[ L^2(\bbT^d) \big]^d \to \big[ L^2(\bbT^d) \big]^d$ is a bounded linear operator. To see that $\bbG$ is invertible, it suffices to show that the symmetric matrix $\bG$ is positive definite, i.e.\ there exists a constant $\gamma > 0$ such that
\begin{equation}\label{eq:GisSPD}
\Vint{\boldsymbol{\rmG}(\bq) {\bfs \eta}, {\bfs \eta}} \geq \gamma \quad \text{ for all } {\bfs \eta} \in \bbS^{d-1}\,, \text{ uniformly in } \bq\,.
\end{equation}
To prove this, first note that 
\begin{equation}\label{eq:GisSPDProof}
\Vint{\boldsymbol{\rmG}(\bq) \bfeta, \bfeta} = \intdm{\bbR^d}{\rho(\bq-\by) {\mu_s}(\by,\bq) \left( \frac{\bq-\by}{|\bq-\by|} \cdot {\bfs \eta} \right)^2}{\by} \geq \alpha_1 \intdm{\bbR^d}{\rho(\bz) \left( \frac{\bz}{|\bz|} \cdot {\bfs \eta} \right)^2}{\bz}\,.
\end{equation}
Since $\text{meas}(\supp \rho) > 0$, the quantity $\left( \frac{\bz}{|\bz|} \cdot {\bfs \eta} \right)$ cannot be identically zero on $\supp \rho$. Therefore for each ${\bfs \eta} \in \bbS^{d-1}$ we have $\Vint{\boldsymbol{\rmG}(\bq) {\bfs \eta}, {\bfs \eta}} > 0$. Since the mapping ${\bfs \eta} \mapsto \Vint{\boldsymbol{\rmG}(\bq) {\bfs \eta}, {\bfs \eta}}$ is continuous on a compact set, it follows that $\Vint{\boldsymbol{\rmG}(\bq) {\bfs \eta}, {\bfs \eta}} \geq \gamma$ for all ${\bfs \eta} \in \bbS^{d-1}$ with lower bound independent of $\bq$ by \eqref{eq:GisSPDProof}.
\end{proof}

\begin{proposition}\label{cpt-K} Assume that $\rho, {\mu}$ and $\lambda$ satisfy \eqref{eq:A1}, \eqref{eq:A2} and \eqref{eq:A3}. The operator 
$\bbK : \big[ L^2(\bbT^d) \big]^d \to \big[ L^2(\bbT^d) \big]^d$ is a well-defined  bounded, linear, and compact operator. Moreover, we have 
\begin{equation*}
\Vnorm{\bbK {\bfs \psi}}_{L^2(\bbT^d)} \leq \alpha_2 \|\rho\|_{L^{1}}  \Vnorm{{\bfs \psi}}_{L^2(\bbT^d)},\quad \forall {\bfs \psi}\in \big[ L^2(\bbT^d) \big]^d\,.
\end{equation*}
\end{proposition}

\begin{proof}
Given ${\bfs \psi}  \in \big[ L^1(\bbT^d) \big]^d$, we first show that $(\bq, \bz)\mapsto{\bfs \rmK}(\bz) {\mu_s}(\bq, \bq-\bz) {\bfs \psi}(\bq-\bz)\in \big[L^{1}(\bbT^d \times \bbR^{d})\big]^{d}$.  This follows from Tonelli's theorem and Young's inequality since by  periodicity we have 
\begin{equation*}
\begin{split}
\iint_{\bbT^d \times \bbR^{d}}|{\rmK}(\bz) {\mu_s}(\bq, \bq-\bz) {\bfs \psi}(\bq-\bz)| \, \rmd\by \, \rmd\bq 
	&\leq \alpha_2 \iint_{\bbT^d \times \bbR^{d}}\rho(\bz)|{\bfs \psi}(\bq-\bz)| \, \rmd\bq \, \rmd\bz \\
	&=\alpha_2 \int_{\bbR^d} \rho(\bz)\left(\int_{ \bbT^{d}}|{\bfs \psi}(\bq-\bz)| \, \rmd\bq \right) \, \rmd \bz \\
	&= \alpha_2 \|{\bfs \psi}\|_{L^{1}(\bbT^{d})} \|\rho\|_{L^1(\bbR^{d})} < \infty\,.
\end{split}
\end{equation*}\
As a consequence, for almost all $\bq\in \bbT^{d}$, $\boldsymbol{\rmK}(\cdot) {\mu_s}(\bq, \bq-\cdot) {\bfs \psi}(\bq-\cdot)\in \big[L^{1}(\bbR^{d})\big]^{d}$. That is, $\bbK {\bfs \psi}(\bq)$ is well-defined for almost all $\bq \in \bbT^{d}$.  Linearity of the operator is obvious. 
A similar argument also shows that for $\bq\in \bbT^{d}$ and ${\bfs \psi} \in \big[ L^2(\bbT^d) \big]^d$ the function $\rho(\bq-\cdot) |{\bfs \psi}|^{2} \in L^{1}(\mathbb{R}^{d})$. 
Now for any ${\bfs \psi} \in \big[ L^2(\bbT^d) \big]^d$, we have after a change of variables, H\"older's inequality, and Tonelli's theorem
\begin{equation*}
\begin{split}
\Vnorm{\bbK {\bfs \psi}}_{L^2(\bbT^d)}^2 &= \intdm{\bbT^d}{ \left| \intdm{\bbR^d}{\boldsymbol{\rmK}(\bq-\by) {\mu_s}(\bq, \by) {\bfs \psi}(\by)}{\by}\right|^{2}} {\bq} \\
&=  \intdm{\bbT^d}{ \left| \intdm{\bbR^d}{\boldsymbol{\rmK}(\bz) {\mu_s}(\bq, \bq-\bz) {\bfs \psi}(\bq-\bz)}{\bz}\right|^{2}} {\bq}\\
&\leq \alpha_{2}^{2}\intdm{\bbT^d}{ \left| \intdm{\bbR^d}{\rho(\bz)  {\bfs \psi}(\bq-\bz)}{\bz}\right|^{2}}{\bq}\\
& \leq \alpha^2_{2}\|\rho\|_{L^{1}(\bbR^d)} \intdm{\bbT^d}{ \intdm{\bbR^d}{\rho(\bz)  |{\bfs \psi}(\bq-\bz)|^2}{\bz}}{\bq}\\
& = \alpha^2_{2}\|\rho\|_{L^{1}(\bbR^d)} \intdm{\bbR^d}{ \rho(\bz)\intdm{\bbT^d}{  |{\bfs \psi}(\bq-\bz)|^2}{\bq}}{\bz}
	= \alpha^2_{2}\|\rho\|_{L^{1}(\bbR^d)}^{2} \|{\bfs \psi}\|^{2}_{L^{2}(\bbT^{d})}\,.\\
\end{split}
\end{equation*}

To show compactness of the operator, we show that $\bbK$ is a uniform limit of bounded, compact linear operators. To that end, let $N \in \bbN$ and define
\begin{equation*}
\bbK_N {\bfs \psi}(\bq) := \intdm{\bbR^d}{\boldsymbol{\rmK}_N(\bq-\by) {\mu_s}(\bq, \by) {\bfs \psi}(\by)}{\by}\, \quad \text{for $\bq\in \bbT^{d}$} ,
\end{equation*}
where $\boldsymbol{\rmK}_N(\bz) := \rho(\bz) \chi_{[-N,N]^d}(\bz) \frac{\bz \otimes \bz}{|\bz|^2}\,$ is a matrix with bounded support.  
We claim that for each $N$, the operator $\bbK_N : \big[L^2(\bbT^d) \big]^d \to \big[ L^2(\bbT^d) \big]^d$ is compact.
Assuming this and applying the previous estimate for the operator $\bbK - \bbK_N$, 
\begin{equation*}
\Vnorm{\bbK - \bbK_N}_{\cL( [L^2(\bbT^d)]^d \, , \, [L^2(\bbT^d)]^d )} \leq \alpha_2 \Vnorm{\rho- \rho \chi_{[-N,N]^d}(\cdot)}_{L^1(\bbR^d)} \to 0 \quad \text{ as } N \to \infty\,,
\end{equation*}
and thus $\bbK$ is compact as it is the limit (in the operator norm) of compact operators $\bbK_N$.

To see that $\bbK_N$ is compact, we use the periodicity of ${\mu_s}$ and ${\bfs \psi}$ to write 
\begin{equation*}
\begin{split}
\bbK_N {\bfs \psi}(\bq) &= \intdm{\bbR^d}{\boldsymbol{\rmK}_N(\bq-\by){\mu_s}(\bq, \by) {\bfs \psi}(\by)}{\by} \\
	&= \sum_{\bk \in \bbZ^d} \intdm{\bbT^d}{\boldsymbol{\rmK}_N(\bq-\by+\bk){\mu_s}(\bq, \by) {\bfs \psi}(\by)}{\by}\,. \\
%	&= \sum_{\bk \in \bbZ^d} \intdm{\bbT^d}{\rho(\bq-\by+\bk) \chi_{[-N,N]^d}(\bq-\by+\bk) {\mu_s}(\bq, \by) \frac{(\bq-\by+\bk) \otimes (\bq-\by+\bk)}{|\bq-\by+\bk|^2} {\bfs \psi}(\by)}{\by}\,. \\
\end{split}
\end{equation*}
Writing ${\mu_s}(\bq, \by) = {\mu(\bq)\over 2}  + {\mu(\by)\over 2}$ and using that $\supp \bK_N \subseteq [-N,N]^d$,
%$\by$, $\bq \in \bbT^d$ and $\bk \in \bbZ^d$ we have that 
%\begin{equation*}
%\begin{split}
%\bbK_N {\bfs \psi}(\bq) &= \sum_{\substack{\bk \in \bbZ^d \\ \bk \in [-N-2,N+2]^d}} \intdm{\bbT^d}{\rho(\bq-\by+\bk) \chi_{[-N,N]^d}(\bq-\by+\bk) {\mu_s}(\bq, \by) \frac{(\bq-\by+\bk) \otimes (\bq-\by+\bk)}{|\bq-\by+\bk|^2} {\bfs \psi}(\by)}{\by}\,\\
%&={\mu(\bq)\over 2}\sum_{\substack{\bk \in \bbZ^d \\ \bk \in [-N-2,N+2]^d}} \intdm{\bbT^d}{\rho(\bq-\by+\bk) \chi_{[-N,N]^d}(\bq-\by+\bk) \frac{(\bq-\by+\bk) \otimes (\bq-\by+\bk)}{|\bq-\by+\bk|^2} {\bfs \psi}(\by)}{\by}\\
%&+\sum_{\substack{\bk \in \bbZ^d \\ \bk \in [-N-2,N+2]^d}} \intdm{\bbT^d}{\rho(\bq-\by+\bk) \chi_{[-N,N]^d}(\bq-\by+\bk) {{\mu}(\by)\over 2} \frac{(\bq-\by+\bk) \otimes (\bq-\by+\bk)}{|\bq-\by+\bk|^2} {\bfs \psi}(\by)}{\by}
%\end{split}
%\end{equation*}
\begin{equation*}
\begin{split}
\bbK_N {\bfs \psi}(\bq) &= \sum_{\substack{\bk \in \bbZ^d \\ \bk \in [-N-2,N+2]^d}} \intdm{\bbT^d}{\bK_N(\bq-\by+\bk) \mu_s(\bq,\by) {\bfs \psi}(\by)}{\by} \\
	&={\mu(\bq)\over 2}\sum_{\substack{\bk \in \bbZ^d \\ \bk \in [-N-2,N+2]^d}} \intdm{\bbT^d}{\bK_N(\bq-\by+\bk) {\bfs \psi}(\by)}{\by} \\
	&\qquad +\sum_{\substack{\bk \in \bbZ^d \\ \bk \in [-N-2,N+2]^d}} \intdm{\bbT^d}{ \bK_N(\bq-\by+\bk) {{\mu}(\by)\over 2} {\bfs \psi}(\by)}{\by}\,.
\end{split}
\end{equation*}

We will show that $\bbK_N {\bfs \psi}$ consists of --  in the appropriate sense -- a finite linear combination of convolution and multiplication operators. This fact must be shown carefully, as ${\bfs \psi}$ is periodic and thus does not necessarily belong to a Lebesgue space defined over all of $\bbR^d$. With this goal in mind, define $\boldsymbol{\rmK}_{N,k}$ as $\boldsymbol{\rmK}_{N,k}(\bz) := \boldsymbol{\rmK}_N(\bz+ \bk)$, and define 
\begin{equation*}
\widetilde{{\mu}}(\by) :=
	\begin{cases}
	{\mu}(\by) & \by \in \bbT^d\,, \\
	0 & \by \in \bbR^{d} \setminus \bbT^d\,,
	\end{cases}
\qquad
\widetilde{{\bfs \psi}}(\bq) :=
	\begin{cases}
	{\bfs \psi}(\bq) & \bq \in \bbT^d\,, \\
	0 & \bq \in \bbR^d \setminus \bbT^d\,.
	\end{cases}
\end{equation*}
Then we can write $\bbK_N$ as
\begin{equation*}
\bbK_N {\bfs \psi}(\bq) =  \sum_{\substack{\bk \in \bbZ^d \\ \bk \in [-N-2,N+2]^d}} \frac{1}{2} S_{N,k} {\bfs \psi} (\bq)+ \sum_{\substack{\bk \in \bbZ^d \\ \bk \in [-N-2,N+2]^d}} {1\over 2}T_{N,k} {\bfs \psi} (\bq)\,,
\end{equation*}
where the operators $S_{N,k}$, $T_{N,k} : \big[ L^2(\bbT^d) \big]^d \to \big[ L^2(\bbR^d) \big]^d$ are defined as
$T_{N,k}({\bfs \psi})(\bq) := \boldsymbol{\rmK}_{N,k} \ast (\rmM_{\widetilde{{\mu}}} \widetilde{{\bfs \psi}})(\bq)$ 
and $S_{N,k}({\bfs \psi})(\bq) := \rmM_{\widetilde{\mu}} \big( \boldsymbol{\rmK}_{N,k} \ast \widetilde{{\bfs \psi}} \big)(\bq)$.

\noindent Now, let $\cI : \big[ L^2(\bbT^d) \big]^d \to \big[ L^2(\bbR^d) \big]^d$ be the operator that maps ${\bfs \psi}$ to $\widetilde{{\bfs \psi}}$, and suppose $\cG \subset \big[ L^2(\bbT^d) \big]^d$ is a uniformly bounded set. Then the set
\begin{equation*}
\overline{\cG} := (\rmM_{\widetilde{{\mu}}} \cI) (\cG) = \left\lbrace \overline{{\bfs \psi}} \, : \, \overline{{\bfs \psi}} = \rmM_{\widetilde{{\mu}}} \widetilde{{\bfs \psi}}\,, \quad {\bfs \psi} \in \cG \right\rbrace
\end{equation*} 
is clearly a uniformly bounded family in $\big[ L^2(\bbR^d) \big]^d$, since $\Vnorm{\overline{{\bfs \psi}}}_{L^2(\bbR^d)} \leq \alpha_2 \Vnorm{{\bfs \psi}}_{L^2(\bbT^d)}$.
By the Frechet-Kolmogorov Theorem  \cite[Corollary 4.28]{Brezis-book}, the set $(\boldsymbol{\rmK}_{N,k} \ast \overline{\cG}) \big|_{\bbT^d}$ is precompact in $\big[ L^2(\bbT^d) \big]^d$. Therefore, for each $N$ and $\bk  \in [-N-2,N+2]^d \cap \bbZ^d$ the operator $T_{N,k} : \big[ L^2(\bbT^d) \big]^d \to \big[ L^2(\bbT^d) \big]^d$ is compact. In a similar fashion we can prove $S_{N,k} : \big[ L^2(\bbT^d) \big]^d \to \big[ L^2(\bbT^d) \big]^d$ is compact. Thus, $\bbK_N$ is a compact operator, being the linear combination of finitely many compact operators $T_{N,k}$ and $S_{N_k}$.
\end{proof}

Finally, we state some basic results related to periodic functions that we will use in the paper. 
We begin with the following lemma which is proved in \cite[Proposition 7]{Piatnitski-Zhizhina}. 
\begin{lemma}\label{lem:even-odd}
Suppose $g$ and $h$ are bounded periodic functions defined on all of $\bbR^d$. 
Suppose also that two functions $a$ and $b$ belong to $L^1(\bbR^d)$ ans suppose that $a$ is even and $b$ is odd. Then
\begin{equation}\label{eq:aEvenFxn}
\iintdm{\bbR^d}{\bbT^d}{a(\bx-\by)g(\by) h(\bx)}{\bx}{\by} = \iintdm{\bbR^d}{\bbT^d}{a(\bx-\by) g(\bx) h(\by)}{\bx}{\by}\,,
\end{equation}
and
\begin{equation}\label{eq:bOddFxn}
\iintdm{\bbR^d}{\bbT^d}{b(\bx-\by) g(\by)h(\bx)}{\bx}{\by} = - \iintdm{\bbR^d}{\bbT^d}{b(\bx-\by)g(\bx) h(\by)}{\bx}{\by}\,.
\end{equation}
\end{lemma}
Another estimate that we will need in the paper is the following. 
\begin{lemma}\label{lem:darboux}
Suppose that $h\in L^{2}(\bbT^{d})$ and $\psi\in C(\bbR^{d})\cap L^{2}(\bbR^{d})$. Then we have 
\begin{equation*}
\limsup_{\veps\to 0}\sup_{\bq \in \bbR^d} \Vnorm{\left|h \left( \frac{\cdot}{\veps} \right) \right| \, \left| \psi (\cdot + \bq) \right|}_{L^2(\bbR^d)}  \leq \|h\|_{L^{2}(\bbT^d)}\|\psi\|_{L^{2}(\bbR^d)}. 
\end{equation*}
In particular, for sufficiently small $\veps > 0$, the function $h\left({\bx\over \veps}\right)  \psi(\bx)\in L^{2}(\mathbb{R}^{d})$ and 
\begin{equation*}
\limsup_{\veps\to 0}\|h\left({\cdot \over\veps}\right)  \psi(\cdot) \|_{L^{2}(\bbR^{d})} \leq \|h\|_{L^{2}(\bbT^d)}\|\psi\|_{L^{2}(\bbR^d)}\,.
\end{equation*}
\end{lemma}

\begin{proof}
Fix $\veps >0$ and introduce 
\begin{equation}\label{eq:Phi2IsBdd}
M_{\veps} := \sup_{\bq \in \bbR^d} \Vnorm{\left| h \left( \frac{\cdot}{\veps}  \right) \right| \, \left|\psi (\cdot + \bq) \right|}_{L^2(\bbR^d)}. 
\end{equation}
Now note that any $\bq \in \bbR^d$ can be written as $\bq = \bq_0 + \bq_1$, where $\bq_1 \in \veps \bbT^d := [0,\veps]^d$ and $\bq_0 \in \veps \bbZ^d$.
Then, since $h$ is periodic, for any $\by \in \bbR^d$ and any $\bq \in \bbR^d$
\begin{equation*}
\left| h \left( \frac{\by}{\veps} \right) \right| \, \left| h (\by + \bq) \right| = \left| h \left( \frac{\by}{\veps} \right) \right| \, \left| \psi (\by + \bq_1 + \bq_0) \right| =  \left| h \left( \frac{\by+\bq_0}{\veps} \right) \right| \, \left| \psi (\by + \bq_1 + \bq_0) \right|\,.
\end{equation*}
Thus, since a change of coordinates doesn't change the $L^2$-norm of a function, we have for any $\bq \in \bbR^d$
\begin{equation*}
\begin{split}
\intdm{\bbR^d}{\left| h \left( \frac{\by}{\veps} \right) \right|^2 \, \left| \psi (\by + \bq) \right|^2 }{\by} &=  \intdm{\bbR^d}{\left| h \left( \frac{\by+\bq_0}{\veps} \right) \right|^2 \, \left| \psi (\by + \bq_1 + \bq_0) \right|^2}{\by} \\
	&= \intdm{\bbR^d}{\left| h \left( \frac{\by}{\veps} \right) \right|^2 \, \left| \psi (\by + \bq_1) \right|^2}{\by}\,, \qquad \bq_1 \in \veps \bbT^d\,.
\end{split}
\end{equation*}
Therefore taking the supremum over $\bq \in \veps \bbT^d$ is equivalent to taking the supremum over $\bq \in \bbR^d$ and so 
\begin{equation*}
M_{\veps} = \sup_{\bq \in \veps \bbT^d} \Vnorm{\left| h \left( \frac{\cdot}{\veps} \right) \right| \, \left| \psi (\cdot + \bq) \right|}_{L^2(\bbR^d)}\,.
\end{equation*}
Now, set $\rmI_{\bk}(\veps) = \veps \bk + \veps \bbT^d$, for $\bk \in \bbZ^d$. Then by a change of variables
\begin{equation*}
\begin{split}
M_{\veps} &= \sup_{\bq \in \veps \bbT^d} \sum_{\bk \in \bbZ^d} \intdm{\rmI_{\bk}(\veps)}{\left| h \left( \frac{\by}{\veps} \right) \right|^2 \left| \psi(\by+\bq) \right|^2 }{\by} \\
	&\leq \sum_{\bk \in \bbZ^d} \sup_{\substack{\bq \in \veps \bbT^d \\ \by \in \rmI_{\bk}(\veps)}} \left| \psi(\by+\bq) \right|^2  \intdm{\rmI_{\bk}(\veps)}{\left|h\left( \frac{\by}{\veps} \right) \right|^2 }{\by} \\
	&= \Vnorm{h}_{L^2(\bbT^d)}^2  \left( \veps^d  \sum_{\bk \in \bbZ^d} \sup_{\substack{\bq \in \veps \bbT^d \\ \by \in \rmI_{\bk}(\veps)}} \left| \psi(\by+\bq) \right|^2 \right) := \Vnorm{h}_{L^2(\bbT^d)}^2 \cP_{\veps} \,. \\
\end{split}
\end{equation*}
Note that $\cP_{\veps}$ is the Darboux upper sum (with respect to the $\veps$-grid of $\bbR^d$) of the function $|\psi|^{2}\in C(\bbR^{d})\cap L^{1}(\bbR^{d})$, and so $\cP_{\veps}$ converges to the Darboux integral (and thus the Riemann integral) of $|\psi|^2$ as $\veps \to 0$. 
\end{proof}

\section{Asymptotic analysis}\label{asy}
In this section we will prove the main result of the paper. The proof relies on the existence of tensors that act as correctors.
These correctors will be found by solving a system of auxiliary equations, which we obtain in the next subsection.
 
\subsection{Auxiliary problems}

Assume $\bu \in \big[ C^3(\bbR^d) \big]^d \cap \big[H^3(\bbR^d) \big]^d$ is given. For a third-order tensor $\mathfrak{A} = (a^{ijk}) \in \big[ L^2(\bbT^d) \big]^{d^3}$ and a fourth-order tensor $\mathfrak{B} = (b^{ijk \ell}) \in \big[ L^2(\bbT^d) \big]^{d^4}$ that are yet to be determined, define the function 
  $\bw^{\veps} = (w^{\veps}_1, w^{\veps}_2, \ldots , w^{\veps}_d)$ given by the ansatz 
\begin{equation}\label{eq:Ansatz}
\bw^{\veps}(\bx) = \bu(\bx) + \veps^{ } \mathfrak{A}\left(\frac{\bx}{\veps} \right) D{\bf u}(\bx) + \veps^2  \mathfrak{B}\left(\frac{\bx}{\veps} \right) D^{2}{\bf u}(\bx)\,,
\end{equation}
where we recall that $D{\bf u}$ is the gradient matrix of ${\bf u}$ given by $(D{\bf u})_{ij}:={\p u_{i} \over \p{x_j}}$ and $D^{2}{\bf u}$ is the third-order tensor of second partial derivatives of ${\bf u}$ given by $((D^{2}{\bf u})_{ijk}:={\p^2 u^{i}\over \p{x_k}\p{x_j} } )$. 
Applying the operator $\mathbb{L}^{\veps}$ on $\bw^{\veps}$, we have that  
\begin{equation*}
\begin{split}
\bbL^{\veps} \bw^{\veps}(\bx) &= {1\over \veps^{2}} \lambda(\bx, {\bx\over \epsilon}) \int_{\bbR^{d}} {\bfs {\rmK}}_{\veps}(\bx-\by)  {\mu_s} \left(\frac{\bx}{\veps}, \frac{\by}{\veps} \right)\left(\bw^{\veps}(\by)-\bw^{\veps}(\bx)\right) \, \rmd \bx\\
&= \frac{1}{\veps^{2}} \lambda \left( \bx, \frac{\bx}{\veps} \right)\int_{\bbR^d} \rho \left( \bz \right)  {\mu_s} \left( \frac{\bx}{\veps},\frac{\bx}{\veps} - \bz \right)\bigg\lbrace({\bf u}(\bx-\veps \bz)-{\bf u}(\bx))\cdot{{\bz \over |\bz|} }\\
	&   + \veps \left(\left\langle\mathfrak{A} \left( \frac{\bx}{\veps} - \bz \right) D{\bf u}(\bx-\veps \bz), {\bz\over |\bz|}\right\rangle-\left\langle\mathfrak{A} \left( \frac{\bx}{\veps} \right) D{\bf u}(\bx), {\bz\over |\bz|}\right\rangle\right) \\
	&+ \veps^2 \left(\left\langle\mathfrak{B}\left(\frac{\bx}{\veps} - \bz \right) D^{2}{\bf u}(\bx - \veps \bz), {\bz\over |\bz|}\right\rangle- \left\langle\mathfrak{B}\left(\frac{\bx}{\veps}\right) D^{2}{\bf u}(\bx), {\bz\over |\bz|}\right\rangle \right)  \bigg\rbrace 
	\frac{\bz}{|\bz|} \, \mathrm{d}\bz\,,
\end{split}
\end{equation*}
where we have made the change of variables $\frac{\bx-\by}{\veps} = \bz$. We now use the Taylor expansions
\begin{equation*}
\begin{split}
{\bf u}(\bx-\veps \bz) &= {\bf u}(\bx) -\veps D{\bf u}(\bx) \bz +\veps^2 \int_{0}^{1} D^{2}{\bf u}(\bx-\veps t \bz)\bz\otimes \bz(1-t) \, \rmd t\,,\\
D{\bf u}(\bx-\veps \bz) &= D{\bf u}(\bx) -\veps D^{2}{\bf u}(\bx) \bz +\veps^2 \int_{0}^{1} D^{3}{\bf u}(\bx-\veps t \bz)\bz\otimes \bz (1-t) \, \rmd t\,,\\
D^2{\bf u}(\bx-\veps \bz) &= D^2{\bf u}(\bx)  -\veps \int_{0}^{1} D^{3}{\bf u}(\bx-\veps t \bz) \bz \, \rmd t\,.
\end{split}
\end{equation*}
Substituting these expansions into the formula for $\bbL^{\veps} \bw^{\veps}(\bx)$ and collecting powers of $\veps$ we obtain 
\begin{equation*}
\begin{split}
\bbL^{\veps} & \bw^{\veps}(\bx) = \frac{1}{\veps^{2}} \lambda \left( \bx, \frac{\bx}{\veps} \right)\int_{\bbR^d} \rho \left( \bz \right)  {\mu_s} \left( \frac{\bx}{\veps},\frac{\bx}{\veps} - \bz \right)\bigg\lbrace
-\veps \Vint{ D{\bf u}(\bx) \bz, {\bz\over |\bz|}} \\
	&\qquad \qquad + \veps^2 \int_{0}^{1} \Vint{ D^{2}{\bf u}(\bx-\veps t \bz)\bz\otimes \bz, {\bz \over |\bz|} } (1-t) \, \rmd t\\
	&\qquad + \veps \left\langle\left(\mathfrak{A} \left( \frac{\bx}{\veps} - \bz \right) - \mathfrak{A}\left(\frac{\bx}{\veps} \right)\right)D{\bf u}(\bx),  {\bz\over |\bz|}\right\rangle-\veps^2\left\langle \mathfrak{A} \left( \frac{\bx}{\veps}- \bz\right)D^{2}{\bf u}(\bx) \bz,  {\bz\over |\bz|}\right\rangle\\
	&\qquad +\veps^3  \left\langle \mathfrak{A} \left( \frac{\bx}{\veps}- \bz\right)\int_{0}^{1} D^{3}{\bf u}(\bx-\veps t \bz)\bz\otimes \bz (1-t) \, \rmd t, {\bz\over |\bz|}\right\rangle\\
	&\qquad + \veps^2 \left\langle\left(\mathfrak{B}\left(\frac{\bx}{\veps} - \bz \right) -\mathfrak{B}\left(\frac{\bx}{\veps}\right) \right)D^{2}{\bf u}(\bx), {\bz\over |\bz|}\right\rangle \\
	&\qquad \qquad -\veps^3 \left\langle\mathfrak{B}\left(\frac{\bx}{\veps} - \bz \right) \int_{0}^{1} D^{3}{\bf u}(\bx-\veps t \bz) \bz \, \rmd t, {\bz\over |\bz|}\right\rangle \bigg\rbrace \frac{\bz}{|\bz|} \, \mathrm{d}\bz\,. 
\end{split}
\end{equation*}
We summarize the above calculation by writing the equality in the compact form
\begin{equation}\label{op-ansaz}
\bbL^{\veps} \bw^{\veps}(\bx) = \frac{1}{\veps} \lambda \left( \bx, \frac{\bx}{\veps} \right)\Psi_a\left(D{\bf u}(\bx), \frac{\bx}{\veps}\right) \, + \lambda \left( \bx, \frac{\bx}{\veps} \right)\Phi_b\left(D^{2}{\bf u}(\bx), \frac{\bx}{\veps}\right) \, + {\bfs \vphi}^{\veps}(\bx)\,,
\end{equation}
where the vector-valued function $\Psi_a: \bbR^{d^2}\times\bbT^{d}\to \bbR^{d}$ is given by 
\begin{equation}\label{psi}
\Psi_a(\bbM, {\bfs \xi}):= \int_{\bbR^d} \frac{\rho(\bz)}{|\bz|^2}  \mu_s ({\bfs \xi},{\bfs \xi} - \bz) \Big( \Vint{ \big( \mathfrak{A} \left( {\bfs \xi} - \bz \right) - \mathfrak{A}({\bfs \xi} ) \big) \bbM, \bz} - \Vint{\bbM\bz, \bz}  \Big) \bz \, \rmd \bz\,, 
\end{equation}
and the vector-valued function $\Phi_b: \bbR^{d^{3}}\times\bbT^{d}\to \bbR^{d}$ is given by 
\begin{equation}\label{phi}
\begin{split}
\Phi_b(\mathfrak{M}, {\bfs \xi}) = \int_{\bbR^d} 
\frac{\rho(\bz)}{|\bz|^2}  {\mu_s}( {\bfs \xi},{\bfs \xi} - \bz ) \Big( {1\over 2} & \Vint{ \mathfrak{M}\bz\otimes\bz, \bz} - \Vint{\mathfrak{A} ({\bfs \xi} - \bz)\mathfrak{M}{\bz}, \bz} \\
	&\qquad + \Vint{ \big( \mathfrak{B} ( {\bfs \xi} - \bz ) - \mathfrak{B}({\bfs \xi}) \big) \mathfrak{M}, \bz} \Big) \bz \, \rmd \bz\,.
\end{split}
\end{equation}

The remainder term
%${\bfs \vphi}(\bx) := \bbL^{\veps} \bw^{\veps}(\bx) -\frac{1}{\veps} \lambda ( \bx, \frac{\bx}{\veps})\Psi_a (D{\bf u}(\bx), \frac{\bx}{\veps} ) - \lambda ( \bx, \frac{\bx}{\veps} )\Phi_b (D^{2}{\bf u}(\bx), \frac{\bx}{\veps} )$
${\bfs \vphi}^{\veps}(\bx)$ will be explicitly written later.

 The following lemma gives a set of auxiliary systems that will be used for the construction of the tensors $\mathfrak{A} = (a^{ijk}) \in \big[ L^2(\bbT^d) \big]^{d^3}$ and $\mathfrak{B} = (b^{ijk \ell}) \in \big[ L^2(\bbT^d) \big]^{d^4}$. These tensors will be determined and will completely characterize the special linear maps $\Psi_a$ and $\Phi_b$.  We remark that the integrals in \eqref{psi} and \eqref{phi}  can be shown to be well-defined using a similar argument as in the proof of Proposition \ref{cpt-K}.  
\begin{lemma}\label{lemma-CP}
Assume that $\rho, {\mu_s}$ and $\lambda$ satisfy \eqref{eq:A1}, \eqref{eq:A2} and \eqref{eq:A3}. 
Then there exists a third-order tensor $\mathfrak{A} = (a^{ijk}) \in \big[ L^2(\bbT^d) \big]^{d^3}$, a fourth-order tensor $\mathfrak{B} = (b^{ijk \ell}) \in \big[ L^2(\bbT^d) \big]^{d^4}$, and an elasticity tensor $\mathfrak{C}(\bx)$ satisfying \eqref{eq:ElasticityTensor1}--\eqref{eq:ElasticityTensor2} such that 
\begin{equation*}
\Psi_a(\bbM, {\bfs \xi})={\bfs 0},\quad \forall \, \bbM\in\bbR^{d}\,, \quad  \forall \, {\bfs \xi}\in \bbT^{d}\,,
\end{equation*}
and
\begin{equation*}
\lambda(\bx, {\bfs \xi}) \Phi_b(\mathfrak{M}, {\bfs \xi}) := \mathfrak{C}(\bx)\mathfrak{M}\,, \quad \forall \, \mathfrak{M}\in \bbR^{d^3}\,,
\end{equation*}
where $\Psi_a$ and $\Phi_b$ are defined in \eqref{psi} and \eqref{phi} respectively. That is, the fourth-order tensor of coefficients of the linear map $\mathfrak{M}\mapsto \lambda(\bx, {\bfs \xi}) \Phi_b(\mathfrak{M}, {\bfs \xi}) $ can be made independent of ${\bfs \xi}$. 
\end{lemma}
The proof of this lemma is technically involved, and we postpone it until Section \ref{aux}. We will at present assume the result of the lemma in order to prove Theorem \ref{thm:MainThm}. We will also need a result concerning the compatibility of the expansion \eqref{eq:Ansatz} with the nonlocal operators $\bbL^{\veps}$.

 We will also need the following proposition that establishes the existence of a differential operator that approximates the operator $\bbL^\veps$ over a class of smooth functions.

\begin{proposition}\label{lma:MainLemma}
Assume that $\rho, {\mu_s}$ and $\lambda$ satisfy \eqref{eq:A1}, \eqref{eq:A2} and \eqref{eq:A3}. 
Let $\mathfrak{A} = (a^{ijk}) \in \big[ L^2(\bbT^d) \big]^{d^3}$,  $\mathfrak{B} = (b^{ijk \ell}) \in \big[ L^2(\bbT^d) \big]^{d^4}$, and an elasticity tensor $\mathfrak{C}$ are as given in Lemma \ref{lemma-CP}.  
Then for any  given $\bu \in \big[ C^3(\bbR^d) \big]^d \cap \big[H^3(\bbR^d) \big]^d$ and for the function $\bw^{\veps} = (w^{\veps}_1, w^{\veps}_2, \ldots , w^{\veps}_d)$ defined by
\begin{equation*}
\bw^{\veps}(\bx) = \bu(\bx) + \veps \mathfrak{A}\left(\frac{\bx}{\veps} \right) D{\bf u}(\bx) + \veps^2  \mathfrak{B} \left(\frac{\bx}{\veps} \right) D^{2}\bu(\bx)\,,
\end{equation*}
we have
\begin{equation*}
\lim_{\veps\to 0}\|\bbL^{\veps} \bw^{\veps} -\mathfrak{C}(\bx) D^2\bu \|_{L^{2}(\mathbb{R}^{d})} \to 0
\end{equation*}
\end{proposition}
\begin{proof} Notice that by using Lemma \ref{lem:darboux}, for small $\veps$, the vector field $\bw^\veps\in L^{2}(\bbR^d)$. Moreover,  we have rewritten $\bbL^{\veps} \bw^{\veps} $   in \eqref{op-ansaz} as 
\begin{equation*}
\bbL^{\veps} \bw^{\veps}(\bx) = \frac{1}{\veps} \lambda \left( \bx, \frac{\bx}{\veps} \right)\Psi_a\left(D{\bf u}(\bx), \frac{\bx}{\veps}\right) \, + \lambda \left( \bx, \frac{\bx}{\veps} \right)\Phi_b\left(D^{2}{\bf u}(\bx), \frac{\bx}{\veps}\right) \, + {\bfs \vphi}^{\veps}(\bx)
\end{equation*}
where we now explicitly write 
\begin{equation*}
\begin{split}
{\bfs \vphi}^{\veps} &(\bx) \\
	&=\lambda \left( \bx, \frac{\bx}{\veps} \right)\int_{\bbR^d} \rho \left( \bz \right)  {\mu_s} \left( \frac{\bx}{\veps},\frac{\bx}{\veps} - \bz \right)\bigg\lbrace \int_{0}^{1} \left\langle D^{2}{\bf u}(\bx-\veps t \bz)\bz\otimes \bz, {\bz \over |\bz|} \right\rangle (1-t) \, \rmd t \\
	& \quad -{1\over 2} \left\langle D^{2}{\bf u}(\bx)\bz\otimes \bz, {\bz \over |\bz|} \right\rangle  
	+ \veps  \left\langle \mathfrak{A} \left( \frac{\bx}{\veps}- \bz\right)\int_{0}^{1} D^{3}{\bf u}(\bx-\veps t \bz)\bz\otimes \bz (1-t) \, \rmd t, {\bz\over |\bz|}\right\rangle\\
&\quad -\veps \left\langle\mathfrak{B}\left(\frac{\bx}{\veps} - \bz \right) \int_{0}^{1} D^{3}{\bf u}(\bx-\veps t \bz) \bz \, \rmd t, {\bz\over |\bz|}\right\rangle \bigg\rbrace \frac{\bz}{|\bz|} \, \mathrm{d}\bz\,.
\end{split}
\end{equation*}
By the choice of $\mathfrak{A} = (a^{ijk}) \in \big[ L^2(\bbT^d) \big]^{d^3}$,  $\mathfrak{B} = (b^{ijk \ell}) \in \big[ L^2(\bbT^d) \big]^{d^4}$, and the elasticity tensor $\mathfrak{C}$ given in Lemma \ref{lemma-CP}
we have that for any $\veps>0,$
\begin{equation*}
\Psi_a\left(D{\bf u}(\bx), \frac{\bx}{\veps}\right) = 0\,, \quad\forall \, \bx\in \bbR^{d}\,, 
\end{equation*}
and that 
\begin{equation*}
\lambda \left( \bx, \frac{\bx}{\veps} \right)\Phi_b\left(D^{2}{\bf u}(\bx), \frac{\bx}{\veps}\right)  = \mathfrak{C}(\bx) D^2\bu(\bx), \quad\forall \, \bx\in \bbR^{d}\,.
\end{equation*}
To prove the proposition, it suffices to show that  $\|{\bfs \varphi}^\veps\|_{L^{2}(\bbR^{d})} \to 0$ as $\veps \to 0$ for these tensors $\mathfrak{A}$,  $\mathfrak{B}$, and $\mathfrak{C}$.  
We begin by writing ${\bfs \vphi}^{\veps} = {\bfs \vphi}_{1}^{\veps} +{\bfs \vphi}_{2}^{\veps} + {\bfs \vphi}_{3}^{\veps}$, where
\begin{equation*}
\begin{split}
{\bfs \vphi}_{1}^{\veps}(\bx) &= \lambda \Big( \bx, \frac{\bx}{\veps} \Big) \int_{\bbR^d} \frac{\rho(\bz)}{|\bz|^2}  {\mu_s} \left(\frac{\bx}{\veps}, \frac{\bx}{\veps} - \bz \right) \\
	&\hspace{0.75in} \times  \intdmt{0}{1}{\Vint{\big( \grad^2 \bu(\bx-\veps t \bz) - \grad^2 \bu(\bx) \big)\bz\otimes\bz,\bz}(1-t)}{t}  \,\bz \, \rmd \bz\,, \\
	{\bfs \vphi}_{2}^{\veps}(\bx) &= \veps\lambda \Big( \bx, \frac{\bx}{\veps} \Big) \int_{\bbR^d} \rho \left( \bz \right)  {\mu_s} \left(\frac{\bx}{\veps}, \frac{\bx}{\veps} - \bz \right) \\
	&\hspace{0.75in} \times \left\langle \mathfrak{A} \left( \frac{\bx}{\veps}- \bz\right)\int_{0}^{1} D^{3}{\bf u}(\bx-\veps t \bz)\bz\otimes \bz (1-t) \, \rmd t, {\bz\over |\bz|}\right\rangle \frac{\bz}{|\bz|} \,\mathrm{d}\bz\,, \\
	{\bfs \vphi}_{3}^{\veps}(\bx) &=  - \veps\lambda \Big( \bx, \frac{\bx}{\veps} \Big) \int_{\bbR^d} \rho \left( \bz \right)  {\mu_s} \left(\frac{\bx}{\veps}, \frac{\bx}{\veps} - \bz \right) \\
	&\hspace{0.75in} \times \left\langle\mathfrak{B}\left(\frac{\bx}{\veps} - \bz \right) \int_{0}^{1} D^{3}{\bf u}(\bx-\veps t \bz) \bz \, \rmd t, {\bz\over |\bz|}\right\rangle  \frac{\bz}{|\bz|} \, \mathrm{d}\bz\,.
\end{split}
\end{equation*}
We will bound each separately, and show that each one converges to $0$ as $\veps \to 0$. First, for $R > 0$ to be determined,
\begin{equation*}
{\bfs \vphi}_{1}^{\veps}(\bx) =  \lambda \left( \bx, \frac{\bx}{\veps} \right)\int_{|\bz|\leq R} \cdots \, \rmd \bz +  \lambda \left( \bx, \frac{\bx}{\veps} \right)\int_{|\bz| > R} \cdots \, \rmd \bz := {\bfs \vphi}_{1}^{\veps, (\leq R)}(\bx) + {\bfs \vphi}_{1}^{\veps, (> R)}(\bx)\,.
\end{equation*}
Then taking the $L^2$ norm, the first term can be estimated as
\begin{equation*}
\begin{split}
& \Vnorm{{\bfs \vphi}_{1}^{\veps, (\leq R)}}_{L^2(\bbR^d)} \\
	&\leq \alpha_2^2 \sup_{|\bz|\leq R} \Vnorm{D^2 \bu (\cdot - \veps \bz) - D^2 \bu(\cdot)}_{L^2(\bbR^d)} \left( \intdm{\bbR^d}{|\bz|^2 \rho(\bz)}{\bz} \right) \left( \intdmt{0}{1}{(1-t)}{t} \right) \\
	&= \frac{\alpha_2^2 a_2}{2} \sup_{|\bz|\leq R} \Vnorm{D^2 \bu (\cdot - \veps \bz) - D^2 \bu(\cdot)}_{L^2(\bbR^d)}\,,
\end{split}
\end{equation*}
and additionally using H\"older's inequality and Tonelli's theorem gives an estimate for the second term of
\begin{equation*}
\begin{split}
\Vnorm{{\bfs \vphi}_{1}^{\veps, (> R)}}_{L^2(\bbR^d)} & \leq \alpha_2^2 \Vnorm{D^2 \bu}_{L^2(\bbR^d)} \intdm{|\bz|>R}{|\bz|^2 \rho(\bz)}{\bz}\,.
\end{split}
\end{equation*}
Let $\tau>0$ be arbitrary. Choose  $R$ sufficiently large that 
\begin{equation*}
 \alpha_2^2 \Vnorm{D^2 \bu}_{L^2(\bbR^d)} \intdm{|\bz|>R}{|\bz|^2 \rho(\bz)}{\bz}\, < \tau.
\end{equation*}
Then using the continuity of the integral with respect to translations and the fact that $\bu \in \big[ C^3(\bbR^d) \big]^d \cap \big[H^3(\bbR^d) \big]^d$ 
\begin{equation*}
\begin{split}
\lim_{\veps\to 0} \Vnorm{{\bfs \vphi}_{1}^{\veps}}_{L^2(\bbR^d)} &\leq \lim_{\veps\to 0} \Vnorm{{\bfs \vphi}_{1}^{\veps, (\leq R)}}_{L^2(\bbR^d)}  + \lim_{\veps\to 0} \Vnorm{{\bfs \vphi}_{1}^{\veps, (> R)}}_{L^2(\bbR^d)} \\
&< \lim_{\veps\to 0}\frac{\alpha_2^2 a_2}{2} \sup_{|\bz|\leq R} \Vnorm{D^2 \bu (\cdot - \veps \bz) - D^2 \bu(\cdot)}_{L^2(\bbR^d)} +\tau \\
&= \tau\,,
\end{split}
\end{equation*}
and thus $\lim\limits_{\veps\to 0} \Vnorm{{\bfs \vphi}_{1}^{\veps}}_{L^2(\bbR^d)} = 0$. To show that $\Vnorm{{\bfs \vphi}_{2}^{\veps}}_{L^2(\bbR^d)} \to 0$ we use Minkowski's inequality to estimate it as 
\begin{equation*}
\begin{split}
	&\Vnorm{{\bfs \vphi}_{2}^{\veps}}_{L^2(\bbR^d)} \\
	&\leq  \veps \alpha_{2}^{2}\left[ \int_{\bbR^{d}}\left|\int_{\bbR^d} \rho(\bz)  \Big\langle \mathfrak{A} \Big( \frac{\bx}{\veps}- \bz \Big)\int_{0}^{1} D^{3}{\bf u}(\bx-\veps t \bz)\bz\otimes \bz (1-t) \, \rmd t, {\bz\over |\bz|} \Big\rangle \frac{\bz}{|\bz|} \mathrm{d}\bz\right|^{2} \rmd\bx\right]^{1/2}\\
		&\leq \veps \alpha_2^2 \left( \int_{\bbR^d} \left| \int_{\bbR^d} \rho \left( \bz \right) |\bz|^2 \left| \mathfrak{A} \left( \frac{\bx}{\veps} - \bz \right) \right| \intdmt{0}{1}{ \left|D^3 \bu (\bx-\veps t \bz) \right|  (1-t)}{t} \, \rmd \bz \right|^2 \, \rmd \bx \right)^{1/2} \\
	&\leq \veps \alpha_2^2 \int_{\bbR^d} \left( \intdmt{0}{1}{ \intdm{\bbR^d}{ \rho^2(\bz) |\bz|^4 \left| \mathfrak{A} \left( \frac{\bx}{\veps} - \bz \right) \right|^2 \left| D^3 \bu (\bx-\veps t \bz) \right|^2 }{\bx}  (1-t)^2}{t} \right)^{1/2} \, \rmd \bz \\
	&\leq \veps \alpha_2^2 \sup_{\bz, \bq \in \bbR^d} \Vnorm{ \Big| \mathfrak{A} \left( \frac{\cdot}{\veps} - \bz \right) \Big| \big| D^3 \bu (\cdot-\veps \bz + \bq) \big|}_{L^2(\bbR^d)} \intdm{\bbR^d}{|\bz|^2 \rho(\bz)}{\bz} \sqrt{ \intdmt{0}{1}{(1-t)^2}{t} } \\
	&= \veps  \frac{\alpha_2^2 a_2}{\sqrt{3}} \sup_{\bz, \bq \in \bbR^d} \Vnorm{\left| \mathfrak{A} \left( \frac{\cdot}{\veps} - \bz \right) \right| \, \left| D^3 \bu (\cdot-\veps \bz + \bq) \right|}_{L^2(\bbR^d)}\,\\
	&=\veps  \frac{\alpha_2^2 a_2}{\sqrt{3}} \sup_{\bq \in \bbR^d} \Vnorm{\left| \mathfrak{A} \left( \frac{\cdot}{\veps} \right) \right| \, \left| D^3 \bu (\cdot + \bq) \right|}_{L^2(\bbR^d)}
\end{split}
\end{equation*}
where we have made the change of variables $\by = \bx - \veps \bz$ in the last equality. 
The convergence of $\Vnorm{{\bfs \vphi}_{2}^{\veps}}_{L^2(\bbR^d)}$ to $0$ is therefore assured so long as
\begin{equation*}
 \sup_{\bq \in \bbR^d} \Vnorm{\left| \mathfrak{A} \left( \frac{\cdot}{\veps}  \right) \right| \, \left| D^3 \bu (\cdot + \bq) \right|}_{L^2(\bbR^d)} \text{ is bounded uniformly in } \veps\,,
\end{equation*}
which indeed holds as a result of Lemma \ref{lem:darboux}. 
The final quantity ${\bfs \vphi}_{3}^{\veps}$ converges to $0$ by similar reasoning;
\begin{equation*}
\begin{split}
	\Vnorm{{\bfs \vphi}_{3}^{\veps}}_{L^2(\bbR^d)} 
	&\leq \veps \alpha_2^2 \left( \int_{\bbR^d} \left| \int_{\bbR^d} \rho \left( \bz \right)  \left| \mathfrak{B} \left( \frac{\bx}{\veps} - \bz \right) \right| \intdmt{0}{1}{ \left|D^3 \bu (\bx-\veps t \bz) \right|  \, |\bz|}{t} \, \rmd \bz \right|^2 \, \rmd \bx \right)^{1/2} \\
	&\leq \veps \alpha_2^2 \int_{\bbR^d} \left( \intdmt{0}{1}{ \intdm{\bbR^d}{ \rho^2(\bz) |\bz|^2 \left| \mathfrak{B} \left( \frac{\bx}{\veps} - \bz \right) \right|^2 \left| D^3 \bu (\bx-\veps t \bz) \right|^2 }{\bx}}{t} \right)^{1/2} \, \rmd \bz \\
	&\leq \veps \alpha_2^2 \sup_{\bz, \bq \in \bbR^d} \Vnorm{\left| \mathfrak{B} \left( \frac{\cdot}{\veps} - \bz \right) \right| \, \left| D^3 \bu (\cdot-\veps \bz + \bq) \right|}_{L^2(\bbR^d)} \intdm{\bbR^d}{|\bz| \rho(\bz)}{\bz} \\
	&\leq \veps \alpha_2^2 \sqrt{a_1 a_2} \sup_{\bz, \bq \in \bbR^d} \Vnorm{\left| \mathfrak{B} \left( \frac{\cdot}{\veps} - \bz \right) \right| \, \left| D^3 \bu (\cdot-\veps \bz + \bq) \right|}_{L^2(\bbR^d)}\,,
\end{split}
\end{equation*}
where H\"older's inequality was used in the final step. Now we proceed exactly as we did for ${\bfs \vphi}_{2}^{\veps}$, using Lemma \ref{lem:darboux} to demonstrate that the latter quantity goes to $0$ as $\veps\to 0$. 
\end{proof}

\subsection{Proof of main theorem}
In this subsection we use Lemma \ref{lemma-CP} and Proposition \ref{lma:MainLemma} to prove the main result of the paper, Theorem \ref{thm:MainThm}. The proof follows the program laid out in \cite{Piatnitski-Zhizhina} adjusted to our setting. Assume that $\rho, {\mu}$ and $\lambda$ satisfy \eqref{eq:A1}, \eqref{eq:A2} and \eqref{eq:A3}. 
Let $\mathfrak{A} = (a^{ijk}) \in \big[ L^2(\bbT^d) \big]^{d^3}$,  $\mathfrak{B} = (b^{ijk \ell}) \in \big[ L^2(\bbT^d) \big]^{d^4}$, and an elasticity tensor $\mathfrak{C}$ are as given in Lemma \ref{lemma-CP}. We prove Theorem \ref{thm:MainThm} in three steps. In \textbf{Step 1} and \textbf{Step 2} we prove Theorem \ref{thm:MainThm} for $\bff \in \big[ \cS(\bbR^d) \big]^d$, the space of Schwartz functions. \textbf{Step 3} extends this result to a general $\bff \in \big[ L^2(\bbR^d) \big]^d$.

{\bf Step 1.} For  $\bff \in \big[ \cS(\bbR^d) \big]^d$, define $\bu^0$ by \eqref{eq:u0Def}. 
Then since $\bbL^0$ is a second-order strongly elliptic operator with smooth coefficients, the solution $\bu^0$ will also be smooth and at least 
$\bu^0 \in \big[ C^3(\bbR^d) \big]^d\cap \big[ H^3(\bbR^d) \big]^d$. Thus the assumptions of Proposition \ref{lma:MainLemma} are satisfied, and we define the perturbation $\bv^{\veps}$ of $\bu^0$ by the ansatz
\begin{equation}\label{eq:AnsatzMainThm}
\bv^{\veps}(\bx) = \bu^0(\bx) + \veps \mathfrak{A}\left(\frac{\bx}{\veps} \right) D{\bf u}^0(\bx) + \veps^2  \mathfrak{B} \left(\frac{\bx}{\veps} \right) D^{2}\bu^0(\bx)\,.
\end{equation}
Then  as $\veps \to 0$ we have  
\begin{equation}\label{eq:AnsatzProof0}
\Vnorm{\bu^0 - \bv^{\veps}}_{L^2(\bbR^d)} \to 0\,.
\end{equation}
This follows from the application of Lemma \ref{lem:darboux} that guarantees 
\begin{equation}\label{eq:AnsatzProof1}
\Vnorm{ \left| \mathfrak{A} \left(\frac{\cdot}{\veps} \right) \right| \, |\grad \bu^0(\cdot)|}_{L^2(\bbR^d)} \qquad \text{ and } \qquad \Vnorm{\left| \mathfrak{B} \left(\frac{\cdot}{\veps} \right) \right| \, \left|  \grad^2 \bu^0(\cdot) \right|}_{L^2(\bbR^d)}
\end{equation}
are bounded uniformly in $\veps$.

{\bf Step 2.} 
Let $\bff \in \big[ \cS(\bbR^d) \big]^d$. Define $\bu^{\veps} = (m\bbI - \bbL^\veps)^{-1}\bff$. For $\bv^{\veps}$ given by \eqref{eq:AnsatzMainThm} it follows that 
\begin{equation*}
\Vnorm{\bu^{\veps} - \bv^{\veps}}_{L^2(\bbR^d)} \to 0
\end{equation*}
as $\veps \to 0$.
To prove this, we begin by noting that by Proposition \ref{lma:MainLemma},
\begin{equation*}
\bbL^{\veps} \bv^{\veps} = \bbL^0 \bu^0 + {\bfs \varphi}_{\veps}\,, \qquad \Vnorm{{\bfs \varphi}_{\veps}}_{L^2(\bbR^d)} \rarrowop_{\veps \to 0} 0\,.
\end{equation*}
Then we have that 
\begin{equation*}
(\bbL^{\veps} - m \bbI)\bv^{\veps} + m (\bv^{\veps} - \bu^0) = (\bbL^0 - m \bbI)\bu^0 + {\bfs \varphi}_{\veps}\,, \qquad \Vnorm{{\bfs \varphi}_{\veps}}_{L^2(\bbR^d)} \rarrowop_{\veps \to 0} 0\,.
\end{equation*}
By \eqref{eq:AnsatzProof0}, the quantity $m \Vnorm{\bv^{\veps} - \bu^0}_{L^2(\bbR^d)} \to 0$ as $\veps \to 0$, so
\begin{equation}\label{eq:MainThmSchFxn1}
\begin{split}
(\bbL^{\veps} - m \bbI)\bv^{\veps} &= (\bbL^0 - m \bbI)\bu^0 + \widetilde{{\bfs \varphi}}_{\veps} \\
	&= \bff + \widetilde{{\bfs \varphi}}_{\veps}\,,
\qquad \qquad \text{ where } \Vnorm{\widetilde{{\bfs \vphi}}_{\veps}}_{L^2(\bbR^d)} \rarrowop_{\veps \to 0} 0\,.
\end{split}
\end{equation}
We therefore use \eqref{eq:MainThmSchFxn1} to obtain the equation 
\begin{equation*}
\bu^{\veps} = (m \bbI - \bbL^{\veps})^{-1}\bff = (m \bbI - \bbL^{\veps})^{-1} \Big( (m\bbI-\bbL^{\veps})\bv^{\veps} - \widetilde{{\bfs \varphi}}_{\veps} \Big) = \bv^{\veps} - (m \bbI - \bbL^{\veps})^{-1} \widetilde{{\bfs \varphi}}_{\veps}\,.
\end{equation*}
By \eqref{eq:AprioriEst}
\begin{equation*}
\sup_{\veps>0}\Vnorm{(m \bbI - \bbL^{\veps})^{-1}}_{\cL( [L^2(\bbR^d)]^d \, , \, [L^2(\bbR^d)]^d )} \leq \widetilde{C}\,,
\end{equation*}
and therefore
\begin{equation*}
\Vnorm{\bu^{\veps}-\bv^{\veps}}_{L^2(\bbR^d)} = \Vnorm{(m \bbI - \bbL^{\veps})^{-1} \widetilde{{\bfs \varphi}}_{\veps}}_{L^2(\bbR^d)} \leq \widetilde{C} \Vnorm{\widetilde{{\bfs \varphi}}_{\veps}}_{L^2(\bbR^d)} \rarrowop_{\veps \to 0} 0\,.
\end{equation*}
Combining this with \eqref{eq:AnsatzProof0} gives
\begin{equation}\label{eqn:MainThmSchFxn}
\Vnorm{\bu^{\veps}-\bu^0}_{L^2(\bbR^d)} \to 0 \qquad \text{ as } \veps \to 0
\end{equation}
for  every $\bff \in \big[ \cS(\bbR^d) \big]^d$.

{\bf Step 3.}
Finally, let $\bff \in \big[ L^2(\bbR^d) \big]^d$ and $\bu^{\veps} = (m\bbI-\bbL^{\veps})^{-1}\bff$. Then for any $\delta > 0$ there exists $\bff_{\delta} \in \big[ \cS(\bbR^d) \big]^d$ such that $\Vnorm{\bff - \bff_{\delta}}_{L^2(\bbR^d)} < \delta$. Since $(m\bbI - \bbL^{\veps})^{-1}$ is bounded uniformly in $\veps$, we have for $\bu_\delta^{\veps} := (m\bbI - \bbL^\veps)^{-1}\bff_\delta$ and $\bu^0_{\delta} := (m\bbI-\bbL^0)^{-1}\bff_{\delta}$ that
\begin{equation}\label{eq:MainThmL2Fxn}
\Vnorm{\bu^{\veps}_{\delta} - \bu^{\veps}} \leq \widetilde{C} \delta \quad \text{ and } \quad \Vnorm{\bu^{0}_{\delta} - \bu^{0}} \leq \widetilde{C} \delta\,.
\end{equation}
Since $\Vnorm{\bu^{\veps}_{\delta}-\bu^0_{\delta}}_{L^2(\bbR^d)} \to 0$ as  $\veps \to 0$ by \eqref{eqn:MainThmSchFxn}, it follows from  \eqref{eq:MainThmL2Fxn} that
\begin{equation*}
\limsup_{\veps \to 0} \Vnorm{\bu^{\veps} - \bu^0}_{L^2(\bbR^d)} \leq 2 \widetilde{C} \delta
\end{equation*}
for arbitrary $\delta > 0$. Therefore $\lim\limits_{\veps \to 0} \Vnorm{\bu^{\veps} - \bu^0}_{L^2(\bbR^d)} = 0$.

\section{Solvability of the auxiliary system of equations}\label{aux}
In this section, we prove Lemma \ref{lemma-CP}. 
That is, under the assumption that  $\rho, {\mu_s}$ and $\lambda$ satisfy \eqref{eq:A1}, \eqref{eq:A2} and \eqref{eq:A3}, 
we demonstrate that there exists a third-order tensor $\mathfrak{A} = (a^{ijk}) \in \big[ L^2(\bbT^d) \big]^{d^3}$, a fourth-order tensor $\mathfrak{B} = (b^{ijk \ell}) \in \big[ L^2(\bbT^d) \big]^{d^4}$, and an elasticity tensor $\mathfrak{C}(\bx)$  such that 
\begin{equation}\label{3rd-order}
\Psi_a(\bbM, {\bf q})={\bfs 0},\quad \forall \, \bbM\in\bbR^{d}\,, \quad \forall \, {\bf q}\in \bbT^{d}
\end{equation}
and
\begin{equation}\label{4th-order}
\lambda(\bx, {\bf q}) \Phi_b(\mathfrak{M}, {\bf q}) := \mathfrak{C}(\bx)\mathfrak{M}\quad \forall \, \mathfrak{M}\in \bbR^{d^3}\,,
\end{equation}
where the maps $\Psi_a$ and $\Phi_b$ are given by \eqref{psi} and $\eqref{phi}$ respectively.  
Notice that from their definition for a fixed ${\bf q}\in \bbT^{d}$, both $\Psi_a(\cdot, {\bf q})$ and $\Phi_b(\cdot, {\bf q})$ are linear maps in their respective domains. Moreover, since  $\Psi_a(0, {\bf q})=\Phi_b(\mathfrak{0}, {\bf q})=0$, there exist a third-order tensor $\mathfrak{K}({\bf q})$ and a fourth-order tensor $\widetilde{\mathfrak{C}}(\bx, {\bf q})$ such that  
\begin{equation}\label{eq:ContractionFix}
\begin{split}
\Psi_a(\bbM, {\bq}) &= \mathfrak{K}({\bf q})\bbM,\quad \forall \, \bbM\in \bbR^{d^2},\quad \text{and }\\
\lambda(\bx, {\bf q}) \Phi_b(\mathfrak{M}, {\bf q}) &= \widetilde{\mathfrak{C}}({\bx, \bf q}) \mathfrak{M},\quad \forall \, \mathfrak{M}\in \bbR^{d^3}. 
\end{split}
\end{equation}
Thus, proving \eqref{3rd-order} is equivalent to showing that $ \mathfrak{K}({\bf q})=0$ for all $\bq\in \bbT^{d}$, and proving \eqref{4th-order} is equivalent to choosing $\mathfrak{B}$ appropriately so that $\widetilde{\mathfrak{C}}({\bx, \bf q})$
is independent of ${\bq}$. 

\subsection{Existence of the third-order tensor}
We begin by explicitly writing a formula for the third-order tensor $\mathfrak{K}=(\mathfrak{k}^{ikl})$ defined in \eqref{eq:ContractionFix} that is associated to a given periodic third-order tensor $\mathfrak{A} = (a^{ijk}) \in L^{2}(\bbT^{d})$. A straightforward computation using \eqref{eq:ContractionFix} and \eqref{psi} reveals 
\begin{equation*}
\mathfrak{k}^{i k \ell} (\bq) = \intdm{\bbR^d}{\rho(\bz) {\mu_s}(\bq, \bq-\bz) \left\lbrace - z_k z_{\ell} + z_j \left( a^{jk \ell} \left( \bq - \bz \right) - a^{jk \ell} \left( \bq \right) \right)  \right\rbrace \frac{z_i}{|\bz|^2}}{\bz}\,,
\end{equation*}
which is well-defined for almost all $\bq \in \bbT^{d}$; this can easily be verified by following the same arguments as in the proof of Proposition \ref{cpt-K}.
To find $\mathfrak{A} = (a^{ijk}) \in L^{2}(\bbT^{d})$ such that \eqref{3rd-order} holds we solve  
\begin{equation}\label{eq:PreCellProb}
\intdm{\bbR^d}{\rho(\bz) {\mu_s}(\bq, \bq-\bz) \left\lbrace - z_k z_{\ell} + z_j \left( a^{jk \ell} \left( \bq - \bz \right) - a^{jk \ell} \left( \bq \right) \right)  \right\rbrace \frac{z_i}{|\bz|^2}}{\bz} = 0\,, \quad\forall \, \bq\in \mathbb{T}^{d}\,. 
\end{equation}
Now, for each $k$, $\ell$ define the vector field $\ba^{k \ell} : \bbR^d \to \bbR^d$ by $(\ba^{k \ell}(\bq))_i := a^{i k \ell}(\bq)$. Making the change of variables $\by = \bq - \bz$ and writing \eqref{eq:PreCellProb} in vector form by eliminating the index $i$, we arrive at the strongly-coupled system of equations
\begin{equation}\label{eq:CellProblem}
\begin{split}
\intdm{\bbR^d}{\rho(\bq-\by) {\mu_s}(\bq, \by) \bigg( &\frac{(\bq-\by) \otimes (\bq-\by)}{|\bq-\by|^2} \big( \ba^{k \ell}(\by) - \ba^{k \ell}(\bq) \big) \\
	&- \frac{(q_k - y_k)(q_{\ell}-y_{\ell})}{|\bq - \by|^2} (\bq-\by)\bigg)}{\by} = {\bfs 0}\,,
\end{split}
\end{equation}
Demonstrating existence of a vector field $\ba^{k \ell} \in \big[ L^2(\bbT^d) \big]^d$ satisfying \eqref{eq:CellProblem} for almost all $\bq\in \bbT^{d}$ and for each $k$, $\ell$ will imply \eqref{3rd-order}. To this end, using the operators we have defined in \eqref{eq:MatricesInOperator}, we can rewrite \eqref{eq:CellProblem} as
\begin{equation}\label{eq:CellProblem-operatorform}
(\bbK-\bbG)\ba^{k \ell}=\bh^{k \ell}\,, \qquad \text{on $\bbT^{d}$ for }1 \leq k, \ell \leq d\,,
\end{equation}
where we introduced the function 
\begin{equation}\label{eq:CellProbData}
\bh^{k \ell}(\bq) := \intdm{\bbR^d}{ \rho(\bq-\by) {\mu_s}(\bq, \by) \frac{(q_k - y_k)(q_{\ell}-y_{\ell})}{|\bq - \by|^2} (\bq-\by)}{\by}\,, \qquad \bq \in \bbT^{d}\,.
\end{equation}
Observe that $\bh^{k \ell} \in \big[ L^2(\bbT^{d}) \big]^d$, since ${\mu_s}$ is periodic in both variables, and  $\bh^{k \ell}(\bq)$ is in fact bounded for all $\bq \in \bbT^d$ since by H\"older's inequality 
\begin{equation*}
|\bh^{k \ell}(\bq)| \leq \alpha_2 \intdm{\bbR^d}{\rho(\bq-\by) |\bq-\by|}{\by} \leq \alpha_2 \,\sqrt{a_1 \,a_2} < \infty\,.
\end{equation*}
To prove existence of $\ba^{k \ell} \in \big[ L^2(\bbT^d) \big]^d$ satisfying \eqref{eq:CellProblem-operatorform}, we apply the Fredholm Alternative Theorem. Indeed,  
we have shown in Proposition \ref{cpt-K} that the operator $\bbK: \big[ L^2(\bbT^d) \big]^d \to \big[ L^2(\bbT^d) \big]^d$  is a compact operator, and in Proposition \ref{PD-G} that $\bbG$ is a positive invertible operator. 
By the Fredholm Alternative Theorem, the equation $(\bbG-\bbK)\ba^{k \ell}=-\bh^{k \ell}$ is solvable in $\bbL^{2}(\bbT^{d})$ if and only if $\bh^{k\ell}$ is orthogonal to  all elements of the kernel of the adjoint operator $[\bbK-\bbG]^\ast$. 

\begin{proposition}\label{prop:KerOfK-G}
$\bbK-\bbG$ is self-adjoint and its kernel is the set of constant vector fields.
\end{proposition}

Assuming that this proposition is true for now,  the solvability of $(\bbG-\bbK)\ba^{k \ell}=-\bh^{k \ell}$ 
is equivalent to showing
\begin{equation*}
\intdm{\bbT^d}{\Vint{\bh^{k \ell}(\bq),\bm}}{\bq} = 0\,, \qquad \bm \in \bbR^d\,, \qquad 1 \leq k, \ell \leq d\,.
\end{equation*}
This is indeed the case, using the formula for $\bh^{k \ell}$ above and \eqref{eq:bOddFxn} in Lemma \ref{lem:even-odd}. For each $k$, $\ell$ the choice of $\ba^{k \ell}$ is unique up to addition by a constant vector. In order to fix the choice of $\ba^{k \ell}$ we enforce the condition
\begin{equation}\label{eq:UniquenessCondFora}
\intdm{\bbT^d}{\ba^{k \ell}(\bq)}{\bq} = {\bf 0}\,.
\end{equation}
Note that $\bh^{k \ell} = \bh^{\ell k}$, 
so therefore for each $i$ the matrix $(a^{ik \ell})_{k \ell}$ is symmetric.

\begin{proof}[Proof of Proposition \ref{prop:KerOfK-G}]
For any ${\bfs \psi}$ in  $\big[ L^2(\bbT^d) \big]^d$, we first notice that for $\bq \in \bbT^d$
\begin{equation*}
\begin{split}
(\bbK-\bbG){\bfs \psi} (\bq)&=\intdm{\bbR^d}{\rho(\bq-\by) {\mu_s}(\bq, \by) \frac{(\bq-\by) \otimes (\bq-\by)}{|\bq-\by|^2} \big( {\bfs \psi}(\by) -{\bfs \psi}(\bq) \big)}{\by} \\
&= \intdm{\bbT^d}{\overline{\boldsymbol{\rmK}}(\bq-\by) {\mu_s}(\bq, \by) \big( {\bfs \psi}(\by) - {\bfs \psi}(\bq) \big)}{\by} \,,
\end{split}
\end{equation*}
for $\overline{\boldsymbol{\rmK}}(\bz) := \sum_{\bk \in \bbZ^d} \boldsymbol{\rmK}(\bz+\bk)$ for $ \bz \in \bbT^d\,,$ where we have used the periodicity of ${\bfs \psi}$ and ${\mu_s}$.  The matrix-valued map $\boldsymbol{\rmK}$ is even, making the operator $\bbK-\bbG$ self-adjoint.  Moreover, 
\begin{equation*}
\begin{split}
&\Vint{(\bbK-\bbG) {\bfs \psi}, {\bfs \psi}}_{L^2(\bbT^d)} \\
	&\quad = \iintdm{\bbT^d}{\bbR^d}{\frac{\rho(\bq-\by) {\mu_s}(\bq, \by)}{|\bq-\by|^2} \Big\{ \big( {\bfs \psi}(\bq) - {\bfs \psi}(\by) \big) \cdot (\bq-\by) \big) \Big\} \Big\{ \big( {\bfs \psi}(\bq) \cdot (\bq-\by) \big) \Big\} }{\by}{\bq} \\
	&\quad =\iintdm{\bbT^d}{\bbT^d}{{\mu_s}(\bq, \by)\left\langle {\overline{\boldsymbol{\rmK}}(\bq-\by)  \big( {\bfs \psi}(\by) - {\bfs \psi}(\bq) \big)}, {\bfs \psi}(\bq)\right\rangle}{\by}{\bq}\\
	&\quad ={1\over 2}\iintdm{\bbT^d}{\bbT^d}{{\mu_s}(\bq, \by)\left\langle {\overline{\boldsymbol{\rmK}}(\bq-\by)  \big( {\bfs \psi}(\by) - {\bfs \psi}(\bq) \big)}, {\bfs \psi}(\by)-{\bfs \psi}(\bq)\right\rangle}{\by}{\bq}\\
	&\quad = \frac{1}{2} \iintdm{\bbT^d}{\bbT^d}{{\mu_s}(\bq,\by)\sum_{\bk\in \bbZ^d}\frac{\rho(\bq-\by+\bk) }{|\bq-\by+\bk|^2} \Big( \big( {\bfs \psi}(\bq) - {\bfs \psi}(\by) \big) \cdot (\bq-\by+{\bk}) \Big)^2}{\by}{\bq}\,.
\end{split}
\end{equation*}
One may also rewrite the last equality as 
\begin{equation*}
\Vint{(\bbK-\bbG) {\bfs \psi}, {\bfs \psi}}_{L^2(\bbT^d)} =\frac{1}{2} \iintdm{\bbT^d}{\bbR^d}{{\mu_s}(\bq,\by)\frac{\rho(\bq-\by) }{|\bq-\by|^2} \Big( \big( {\bfs \psi}(\bq) - {\bfs \psi}(\by) \big) \cdot (\bq-\by) \Big)^2}{\by}{\bq}\,.
\end{equation*}
Therefore, if ${\bfs \psi}$ belongs to the kernel of  $(\bbK-\bbG) $, then for almost every $\bq, \by \in \bbT^{d}$  such that $\bq-\by \in \supp \rho$ 
\begin{equation}\label{eq:InfRigidDis}
\big( {\bfs \psi}(\bq) - {\bfs \psi}(\by) \big) \cdot (\bq-\by) = 0\,.
\end{equation}
is satisfied.  We now show that any periodic function that satisfies \eqref{eq:InfRigidDis} must be a constant. We do this in two steps. In the first step we show that ${\bfs \psi}$ must be an infinitesimal rigid map of the form ${\bfs \psi}(\by) = \boldsymbol{\rmQ}\by + \bm$ on $\bbT^d$ where $\boldsymbol{\rmQ}$ is  a skew symmetric matrix and $\bm$ is a vector.  In the second step we show that ${\bfs \rmQ}$ must be the zero matrix. 

{\bf Step 1.} For almost any $\bq$ in the interior of $\bbT^{d}$, set $\delta_{\bq} := \max \left\{ \delta_0\,, \frac{\dist(\bq, \p \bbT^d)}{2} \right\}  >0$. Define $\Gamma(\bq) := \left\{ \by \in B_{\delta_{\bq}}(\bq) \, : \, \frac{\by-\bq}{|\by-\bq|} \in \cJ \cap - \cJ  \right\}$. Then $\Gamma(\bq)$ is an open set compactly contained in $\mathring{\bbT}^d$. In fact, $\Gamma(\bq)$ is the intersection of the ball $B_{\delta_{\bq}}(\bq)$ with the symmetric cone $\Lambda$ centered at $\bq$. Now, let $\eta > 0$ and let $\{ \bw_i \}_{i=1}^d$ denote a basis for $\bbR^d$ contained in $\cJ \cup -\cJ$; such a basis exists since $\cH^{d-1}(\cJ)>0$. Then there exists an $\eta > 0$ small such that $\bq + \eta \bw_i \in \mathring{\bbT}^d$ for every $i$; we will work with $\eta$ in this range from here on. Define for each $i$ the set
$\Gamma(\bq + \eta \bw_i) := \left\{ \by \in B_{\delta_{\bq+\eta \bw_i}}(\bq+\eta \bw_i) \, : \, \frac{\by-\bq - \eta \bw_i}{|\by-\bq - \eta \bw_i|} \in \cJ \cap - \cJ  \right\}$, where $\delta_{\bq + \eta \bw_i} = \frac{\dist(\bq+\eta \bw_i, \p \bbT^d)}{2} > 0$. Then by definition of $\Gamma(\bq)$ and $\Gamma(\bq+\eta \bw_i)$, since the Lebesgue integral is continuous with respect to translations, and since the distance function is continuous, there exists an $\eta_{\bq} > 0$ such that the function
\begin{equation*}
\eta \mapsto \intdm{\bbR^d}{ \chi_{\Gamma(\bq)} \cdot  \chi_{\bbT^d} \cdot \prod_{i=1}^d \chi_{\Gamma(\bq+\eta \bw_i)} }{\bx}
\end{equation*}
is positive. For $\bq \in \mathring{\bbT}^d$ set $\widetilde{\Gamma}(\bq) := \left( \bigcap_{i=1}^d \Gamma(\bq+\eta_{\bq} \bw_i) \right) \cap \Gamma(\bq) \cap \mathring{\bbT}^d$. By the discussion above, $\widetilde{\Gamma}(\bq)$ is an open set of positive Lebesgue measure. 

Now, fix $\bq \in \mathring{\bbT}^d$ (up to a set of measure zero). Then since \eqref{eq:InfRigidDis} holds for almost every $\by \in \widetilde{\Gamma}(\bq) \subset \supp \rho + \bq$, we have
\begin{equation}\label{eq:KernelOfA*:Proof1}
\big( {\bfs \psi}(\bq) - {\bfs \psi}(\by) \big) \cdot (\bq - \by) = 0
\end{equation}
and since $\by \in \Gamma(\bq + \eta_{\bq} \bw_i)$
\begin{equation}\label{eq:KernelOfA*:Proof2}
\big( {\bfs \psi}(\bq+\eta_{\bq} \bw_i)-{\bfs \psi}(\by) \big) \cdot (\bq + \eta_{\bq} \bw_i - \by ) = 0
\end{equation}
for almost every $\by \in \widetilde{\Gamma}(\bq)$.
Therefore, adding and subtracting ${\bfs \psi}(\bq)$ in the first argument of 
\eqref{eq:KernelOfA*:Proof2} and $\bq$ in the second and using 
\eqref{eq:KernelOfA*:Proof1} we see that
$$
\big( {\bfs \psi}(\bq + \eta_{\bq} \bw_i)-{\bfs \psi}(\bq) \big) \cdot (\bq-\by) = - \big( {\bfs \psi}(\bq) - {\bfs \psi}(\by) \big) \cdot \eta_q \bw_i
$$
for almost every $\by \in \widetilde{\Gamma}(\bq)$.
So,
$$
{\bfs \psi}(\by) \cdot \bw_i = \frac{1}{\eta_{\bq}} \Big( \big( {\bfs \psi}(\bq+\eta_{\bq} \bw_i) - {\bfs \psi}(\bq) \big) \cdot (\bq-\by) \Big) + {\bfs \psi}(\bq) \cdot \bw_i
$$
for almost every $\by \in \widetilde{\Gamma}(\bq)$ and for every $i$, so ${\bfs \psi}(\by) \cdot \bw_i$ is clearly a 
linear map.
Then, letting $\boldsymbol{\rmW} = (\bw_i)_i$ be the matrix of basis vectors, and 
${\bfs \psi} = (\psi_1, \psi_2, \ldots, \psi_d)$, we have that 
\begin{align*}
\psi_i(\by) = \left( \boldsymbol{\rmW}^{-1} ( \boldsymbol{\rmW} {\bfs \psi} )\right)_i = \sum_j 
w^{-1}_{ij} ( \bw_j\cdot {\bfs \psi}(\by))
\end{align*}
which, being a sum of of linear maps, is still linear. We conclude that for 
almost all $\by \in \widetilde{\Gamma}(\bq)$ the vector field ${\bfs \psi}(\by)$ is of the form 
$\boldsymbol{\rmQ}_{\bq}\by + \bm_{\bq}$, where $\mathbb{A}$ where $\boldsymbol{\rmQ}_{\bq}$ is a matrix with 
constant entries (depending possibly on $\bq$) and $\bm_{\bq}$ is a constant vector (also depending on $\bq$) in $\bbR^d$. 

Next, given any two points in $\mathring{\bbT}^d$, outside of a set of measure zero, we connect them by  finitely many sets of the form $\widetilde{\Gamma}(\bq)$ i.e.\ for 
any two points $\bq_0$ and $\bp_0$ in $\mathring{\bbT}^d$  there exists a finite chain of $\{ \widetilde{\Gamma}(\bq) \}_{\bq \in \bbT^d}$, denoted 
$\{ \widetilde{\Gamma}(\bq_k) \}_{k=0}^N$, such that $\widetilde{\Gamma}(\bq_k) \cap 
\widetilde{\Gamma}(\bq_{k+1}) \neq \emptyset$ and $\bq_0 \in \widetilde{\Gamma}(\bq_0)$, $\bp_0 \in \widetilde{\Gamma}(\bq_N)$. This is possible, 
since the line segment connecting $\bq_0$ and $\bp_0$ is compact and contained in the convex set $\bbT^d$. Therefore the ${\bfs \psi}$ given above is the same in neighboring intersecting open sets  and so ${\bfs \psi}(\by) 
= \boldsymbol{\rmQ}\by + \bm$ on $\bbT^d$ where $\boldsymbol{\rmQ}$ and $\bm$ are now independent of any point $\bq$ in $\mathring{\bbT}^d$. By \eqref{eq:InfRigidDis}, the matrix $\boldsymbol{\rmQ}$ must be skew symmetric.  
Thus ${\bfs \psi}(\bq)$ has the form of a function belonging to $\cM$ for almost every $\bq \in \bbT^d$.

{\bf Step 2.}  Now we show that $\boldsymbol{\rmQ}$ must be identically zero. Define the translated symmetric cone centered at $\bq$ by $\Lambda(\bq) :=  \left\{ \by \in B_{\delta_0}(\bq) \, : \, \frac{\by-\bq}{|\by-\bq|} \in \cJ \cap - \cJ  \right\}$. Then there exists a $\delta > 0$ such that for $\bq = (1-\delta, 1-\delta, \ldots\, 1-\delta) \in \bbT^d$ the following holds: For each $i \in \{ 1, \ldots, d \}$ the set $B_{\delta_0}(\bq-\be_i) \cap \Lambda(\bq-\be_i) \cap \mathring{\bbT}^d$ is an open nonempty set, and moreover there exists a collection $\{\by^{i,j} \}_{j = 1}^d \subset B_{\delta_0}(\bq-\be_i) \cap \Lambda(\bq-\be_i) \cap \mathring{\bbT}^d$ such that $\{\bq - \by^{i,j} \}_{j = 1}^d$ forms a basis for $\bbR^d$.
Thus, since $\by^{i,j} + \be_i \in \supp \rho + \bq$,
\begin{equation*}
\big( {\bfs \psi}(\bq) - {\bfs \psi}(\by^{i,j}+\be_i) \big) \cdot (\bq-\by^{i,j}-\be_i) = 0\,, \qquad j \in \{ 1, \ldots, d \}\,.
\end{equation*}
Since ${\bfs \psi}$ is periodic,
\begin{equation*}
\big( {\bfs \psi}(\bq) -{\bfs \psi}(\by^{i,j}) \big) \cdot (\bq-\by^{i,j}-\be_i) = 0\,, \qquad j \in \{ 1, \ldots, d \}\,.
\end{equation*}
Now, $\bq$ and $\by^{i,j}$ both belong to $\bbT^d$, so there exists a skew-symmetric matrix $\boldsymbol{\rmQ}$ and a constant vector $\bm \in \bbR^d$ such that ${\bfs \psi}(\bq) = \boldsymbol{\rmQ}\bq + \bm$ in $\bbT^d$. Therefore, since $\boldsymbol{\rmQ}$ is skew-symmetric,
\begin{equation*}
\begin{split}
0 &= \big({\bfs \psi}(\bq) - {\bfs \psi}(\by^{i,j}) \big) \cdot (\bq-\by^{i,j}-\be_i) \\
&= - \big( {\bfs \psi}(\bq) - {\bfs \psi}(\by^{i,j}) \big) \cdot \be_i \\
&= \boldsymbol{\rmQ}(\bq-\by^{i,j}) \cdot \be_i\,.
\end{split}
\end{equation*}
Since $\bq-\by^{i,j}$ is a basis we can write $\bx \in \bbT^d$ as $\sum_{j=1}^d x_j (\bq-\by^{i,j})$, and therefore
\begin{equation*}
\boldsymbol{\rmQ}\bx \cdot \be_i = \sum_{j=1}^d x_j \boldsymbol{\rmQ}(\bq-\by^{i,j}) \cdot \be_i = 0\,,
\end{equation*}
for every $\bx \in \bbT^d$. Since $i$ is arbitrary it follows that $\boldsymbol{\rmQ} \equiv {\bf 0}$.
Therefore ${\bfs \psi}(\bq) \equiv \bm$ for some constant $\bm \in \bbR^d$, completing the proof. 
\end{proof}

\subsection{Existence of the fourth-order tensor}
Now that we have the third-order tensor $\mathfrak{A}$ that will make the map $\Psi_{a}$ the zero map, we will use it to show the existence of a fourth-order tensor $\mathfrak{B}$ that satisfies \eqref{4th-order}.  From the discussion at the beginning of this section, we will find a $\mathfrak{B} = (b^{mjkl}(\bq))$ such that \eqref{4th-order} holds for $\mathfrak{C}(\bx) = (c^{ijkl}(\bx))$ chosen appropriately. In components, we seek $(b^{mjkl}(\bq))$ and $(c^{ijkl}(\bx))$ such that
\begin{equation}\label{eq:PreCellProb2}
\begin{split}
\lambda \left( \bx, \bq \right) \int_{\bbR^d} \rho \left( \bz \right)  {\mu_s} \left(\bq, \bq - \bz \right) & \frac{z_i}{|\bz|^2} \Big\lbrace \frac{1}{2} z_j z_k z_{\ell} - z_j z_m a^{m k \ell}(\bq-\bz) \\
	& + z_m \left( b^{m j k \ell} \left( \bq - \bz \right) - b^{m j k \ell} \left( \bq \right) \right) \Big\rbrace \, \mathrm{d}\bz = c^{ijk \ell} (\bx)\,.
\end{split}
\end{equation}

Now, for each $j$, $k$, $\ell$ define the vector fields $\bb^{j k \ell}$, $\bc^{j k \ell} : \bbR^d \to \bbR^d$ by $(\bb^{j k \ell}(\bq))_i := b^{i j k \ell}(\bq)$ and $(\bc^{j k \ell}(\bq))_i := c^{i j k \ell}(\bq)$.
Making the change of variables $\by = \bq - \bz$ and writing \eqref{eq:PreCellProb2} in vector form by eliminating the index $i$, we arrive at the strongly-coupled system of equations for $\bb^{j k \ell}$
\begin{equation}\label{setup-B}
\begin{split}
&\intdm{\bbR^d}{\rho(\bq-\by) {\mu_s}(\bq,\by) \frac{(\bq-\by) \otimes (\bq-\by)}{|\bq-\by|^2} \big( \bb^{j k \ell}(\by) - \bb^{j k \ell}(\bq) \big)}{\by} \\
	&\qquad = \frac{\bc^{jk \ell}(\bx)}{\lambda(\bx, \bq)} - \frac{1}{2} \intdm{\bbR^d}{ \rho(\bq-\by) {\mu_s}(\bq,\by) \frac{(q_j - y_j)(q_k - y_k)(q_{\ell}-y_{\ell})}{|\bq - \by|^2} (\bq-\by)}{\by} \\
	&\qquad \quad + \intdm{\bbR^d}{\rho(\bq-\by) {\mu_s}(\bq, \by) (q_j-y_j)  \frac{(\bq-\by) \otimes (\bq-\by)}{|\bq-\by|^2} \ba^{k \ell}(\by)}{\by}\,,
\end{split}
\end{equation}
or by using the notation from Section \ref{tools}
\begin{equation}\label{eq:CellProblem2}
(\bbK-\bbG)\bb^{j k \ell}(\bq) = \bg^{j k \ell}(\bx, \bq)\,, \qquad \bq \in \bbT^d\,, \bx \in \bbR^d\,,
\end{equation}
where $\bg^{j k \ell}(\bx, \bq)$ is the expression on the right-hand side of \eqref{setup-B}.
As we have worked previously solving \eqref{eq:CellProblem2} is equivalent to 
\begin{equation}\label{eq:CellProblem2Operator}
\intdm{\bbT^d}{\Vint{\bg^{j k \ell}(\bx, \bq),\bm}}{\bq} = 0\,, \quad \bm \in \bbR^d\,, \quad 1 \leq k, \ell \leq d\,, \quad \bx \in \bbR^d\,.
\end{equation}
By setting $\bm = \be_i$, equation \eqref{eq:CellProblem2Operator} gives the required condition on $ c^{ijk \ell}(\bx)$ 
\begin{equation}\label{FormulaForC2}
\begin{split}
c^{ijk \ell}&(\bx)\intdm{\bbT^d}{\frac{1}{\lambda(\bx, \bq)}}{\bq} \\
	&= \frac{1}{2} \int_{\bbT^d} \intdm{\bbR^d}{ \rho(\bq-\by)  {\mu_s}(\bq,\by) \frac{(q_i-y_i)(q_j - y_j)(q_k - y_k)(q_{\ell}-y_{\ell})}{|\bq - \by|^2}}{\by} \, \rmd \bq \\
	&\qquad - \int_{\bbT^d} \intdm{\bbR^d}{\rho(\bq-\by) {\mu_s}(\bq,\by) \frac{(q_i-y_i)(q_j-y_j)}{|\bq-\by|^2}  (\bq-\by) \cdot \ba^{k \ell}(\by)}{\by} \, \rmd \bq\\
	&=: \widetilde{c}^{ijk \ell}\,.
\end{split}
\end{equation}
Defining the constant fourth-order tensor $\widetilde{\mathcal{C}} = (\widetilde{c}^{ijk \ell} )$ and rewriting \eqref{FormulaForC2}, the equation
\begin{equation*}
\mathfrak{C}(\bx) = \left(\intdm{\bbT^d}{\frac{1}{\lambda(\bx, \bq)}}{\bq}\right)^{-1}\tilde{\mathcal{C}}
\end{equation*}
gives a formula for the fourth-order tensor $\mathfrak{C}(\bx)$ so that  \eqref{eq:CellProblem2} is solvable. This proves the existence of a fourth-order tensor $\mathfrak{C}(\bx)$.

Next, we use this formula \eqref{FormulaForC2} and the symmetry properties of $\mathfrak{A}$  to demonstrate that ${\mathfrak{C}}$ has the symmetries and ellipticity of an elasticity tensor as defined in \eqref{eq:ElasticityTensor1}-\eqref{eq:ElasticityTensor2}. To that end, using \eqref{eq:A3} it suffices to show that the constant tensor $\widetilde{\mathcal{C}}$ satisfies \eqref{eq:ElasticityTensor1} and \eqref{eq:ElasticityTensor2}. The following theorem does exactly that.
\begin{theorem} Let the third-order tensor $\mathfrak{A}$ be given by Lemma \ref{lemma-CP} and $\widetilde{c}^{ijk \ell}$ be given by \eqref{FormulaForC2}. Then
\begin{equation}\label{alternative-for-C}
\begin{split}
\widetilde{c}^{ijk \ell}= \frac{1}{2} \int_{\bbT^d} \int_{\bbR^d} & \frac{\rho(\bq-\by) {\mu_s}(\bq, \by)}{|\bq-\by|^2} \Big( (q_i-y_i)(q_j-y_j) + \big( \ba^{ij}(\bq) - \ba^{ij}(\by) \big) \cdot (\bq-\by) \Big) \\
&\qquad \times \Big( (q_k-y_k)(q_{\ell}-y_{\ell}) + \big( \ba^{k \ell}(\bq) - \ba^{k \ell}(\by) \big) \cdot (\bq-\by) \Big) \, \rmd \by \, \rmd \bq\,.
\end{split}
\end{equation}
Moreover,  $\widetilde{\mathcal{C}}$ is an elasticity tensor satisfying \eqref{eq:ElasticityTensor1} and \eqref{eq:ElasticityTensor2}. 
 \end{theorem}

\begin{proof}Assume first the validity of the alternate expression for $\widetilde{c}^{ijk \ell}$ given in \eqref{alternative-for-C}. We will prove that $\widetilde{\mathcal{C}}$ is an elasticity tensor satisfying \eqref{eq:ElasticityTensor1} and \eqref{eq:ElasticityTensor2}. The symmetries \eqref{eq:ElasticityTensor1} of  $\widetilde{\mathcal{C}}$ follow from the definition \eqref{alternative-for-C} and the symmetry $\ba^{ij} = \ba^{ji}$ demonstrated in Lemma \ref{lemma-CP}.  Next we show the ellipticity \eqref{eq:ElasticityTensor2}. 
For any symmetric matrix $\boldsymbol{\bbW} = (w^{ij})$, a straightforward calculation shows that
\begin{equation*}
\begin{split}
%\Vint{\widetilde{\mathcal{C}}\boldsymbol{\bbW},\boldsymbol{\bbW}} &= \frac{1}{2} \iintdm{\bbT^d}{\bbR^d}{\frac{\rho(\bq-\by)  {\mu_s}(\bq,\by)}{|\bq-\by|^2} \big( w^{ij}(q_i-y_i)(q_j-y_j) + w^{ij} \big( \ba^{ij}(\bq) - \ba^{ij}(\by) \big) \cdot (\bq-\by) \big)^2}{\by}{\bq} \\
	\big\langle \widetilde{\mathcal{C}}\boldsymbol{\bbW},\boldsymbol{\bbW} \big\rangle &= \iintdm{\bbT^d}{\bbR^d}{\frac{\rho(\bz)  {\mu_s}(\bq,\bq-\bz)}{2 |\bz|^2} \big( w^{ij}z_i z_j  + w^{ij} \big( \ba^{ij}(\bq) - \ba^{ij}(\bq-\bz) \big) \cdot (\bz) \big)^2}{\bz}{\bq} \\
	&\geq \frac{\alpha_1}{2} \iintdm{\bbT^d}{\bbR^d}{\frac{\rho(\bz)}{|\bz|^2} \left( w^{ij}z_i z_j + w^{ij} \big( \ba^{ij}(\bq) - \ba^{ij}(\bq-\bz) \big) \cdot \bz \right)^2  }{\bz}{\bq} \\
	&= \frac{\alpha_1}{2} \iintdm{\bbT^d}{\bbR^d}{\frac{\rho(\bz)}{|\bz|^2} \bigg\lbrace \big( w^{ij}z_i z_j \big)^2 + 2 \big( w^{ij}z_i z_j \big) \Big( w^{ij} \big( \ba^{ij}(\bq) - \ba^{ij}(\bq-\bz) \big) \cdot \bz \Big) \\
	&\qquad \qquad \qquad \qquad + \Big( w^{ij} \big( \ba^{ij}(\bq) - \ba^{ij}(\bq-\bz) \big) \cdot \bz \Big)^2  \bigg\rbrace }{\bz}{\bq} \\
	&= \frac{\alpha_1}{2} \iintdm{\bbT^d}{\bbR^d}{\frac{\rho(\bz)}{|\bz|^2} \bigg\lbrace \big( w^{ij}z_i z_j \big)^2 + \Big( w^{ij} \big( \ba^{ij}(\bq) - \ba^{ij}(\bq-\bz) \big) \cdot \bz \Big)^2  \bigg\rbrace }{\bz}{\bq} \\
	&\qquad + 2 \int_{\bbR^d} \frac{\rho(\bz)}{|\bz|^2} \big( w^{ij}z_i z_j \big) \left(  \left( \intdm{\bbT^d}{\big( \ba^{ij}(\bq) - \ba^{ij}(\bq-\bz) \big) }{\bq} \right) \cdot w^{ij} \bz \right) \, \rmd \bz\,. \\
\end{split}
\end{equation*}
By periodicity of $\ba^{ij}$, the last integral is identically zero. Thus,
\begin{equation*}
\begin{split}
\Vint{\widetilde{\mathcal{C}}\boldsymbol{\bbW},\boldsymbol{\bbW}} 
	&\geq \frac{\alpha_1}{2} \iintdm{\bbT^d}{\bbR^d}{\frac{\rho(\bz)}{|\bz|^2} \bigg\lbrace \big( w^{ij}z_i z_j \big)^2 + \Big( w^{ij} \big( \ba^{ij}(\bq) - \ba^{ij}(\bq-\bz) \big) \cdot \bz \Big)^2  \bigg\rbrace }{\bz}{\bq} \\
	&\geq \frac{\alpha_1}{2} \iintdm{\bbT^d}{\bbR^d}{\frac{\rho(\bz)}{|\bz|^2}  \big( w^{ij}z_i z_j \big)^2 }{\bz}{\bq} = \frac{\alpha_1}{2} \intdm{\bbR^d}{\frac{\rho(\bz)}{|\bz|^2}  \left| \Vint{\boldsymbol{\bbW} \bz,\bz} \right|^2 }{\bz}\,.
\end{split}
\end{equation*}
Now, define $\cN := \{ \boldsymbol{\bbW} \, : \, {\bbW} \text{ symmetric, } |\boldsymbol{\bbW}| = 1 \}$.
Then $\cN$ is a compact set of $\bbM_d(\bbR)$, and since $\boldsymbol{\bbW} \mapsto \frac{\alpha_1^2}{2} \int_{\bbR^d} \frac{\rho(\bz)}{|\bz|^2}  \left| \Vint{\boldsymbol{\bbW} \bz,\bz} \right|^2 \, \rmd \bz$ is a continuous function on $\cN$ it suffices to show that $\frac{\alpha_1}{2} \int_{\bbR^d} \frac{\rho(\bz)}{|\bz|^2}  \left| \Vint{\boldsymbol{\bbW} \bz,\bz} \right|^2 \, \rmd \bz > 0$ for every $\boldsymbol{\bbW} \in \cN$. Assume the contrary, that $\int_{\bbR^d} \frac{\rho(\bz)}{|\bz|^2}  \left| \Vint{\boldsymbol{\bbW} \bz,\bz} \right|^2 \, \rmd \bz = 0$ for some $\bbW \in \cN$. Then $\Vint{\bbW \bz,\bz}= 0 $ for every $\bz \in \supp \rho$. We will show that $\bbW \equiv 0$, a contradiction since $\bbW \in \cN$. To begin, note that since $\cJ$ is an open subset of $\bbS^{d-1}$ there exists a vector ${\bfs \nu} \in \bbS^{d-1}$ and a number $\beta \in (0,1)$ such that
\begin{equation*}
\cA := \left\lbrace \bz \in \bbR^d \, : \, \frac{\bz}{|\bz|} \cdot {\bfs \nu} > 1- \beta  \right\rbrace \Subset \Lambda\,.
\end{equation*}
Since $\cA$ is an open cone, we can choose a basis $\{ \bp_j \}$ for $\bbR^d$ such that $\bp_j \in \cA$ for each $j$. Since $\cA$ is a convex set, any convex combination of two vectors $\bp_j$ and $\bp_k$ also belongs to $\cA$, and so for any $\beta \in (0,1)$
\begin{equation*}
\Vint{\bbW \bp_j,\bp_j} = 0\,, \quad \Vint{ \bbW (\beta \bp_j + (1-\beta) \bp_k),(\beta \bp_j + (1-\beta) \bp_k)} = 0\,, \quad j, k \in \{1, \ldots, d \}\,.
\end{equation*}
By polarizing the previous identity and using the fact that $\bbW$ is symmetric,
\begin{equation*}
\begin{split}
0 &= \beta^2 \Vint{\bbW \bp_j, \bp_j} + 2 \beta (1-\beta) \Vint{\bbW \bp_j, \bp_k} + (1-\beta)^2 \Vint{\bbW \bp_k, \bp_k} \\
& = \Vint{\bbW\bp_j,\bp_k}\,, \quad j, k \in \{1, \ldots, d \}\,.
\end{split}
\end{equation*}
Since $\{ \bp_j \}$ is a basis for $\bbR^d$ we conclude that $\bbW \equiv {\bf 0}$. Thus the lower-bound estimate in \eqref{eq:ElasticityTensor2} is proved. To prove the upper bound in \eqref{eq:ElasticityTensor2} we use the formula \eqref{FormulaForC2}. It suffices to show that the double integral defining $\widetilde{c}^{ijk \ell}$ converges absolutely. This is the case; we have
\begin{equation*}
\begin{split}
| \widetilde{c}^{ijk \ell} | & \leq  \frac{\alpha_2}{2} \int_{\bbT^d} \intdm{\bbR^d}{ \rho(\bq-\by)  |\bq - \by|^2}{\by} \, \rmd \bq + \alpha_2 \hspace{-2pt} \int_{\bbT^d} \intdm{\bbR^d}{\rho(\bq-\by)  |\bq-\by| |\ba^{k \ell}(\by)|}{\by} \, \rmd \bq \\
	&= \frac{a_2 \alpha_2}{2} + \alpha_2 \int_{\bbR^d} |\bz| \rho(\bz) \int_{\bbT^d} |\ba^{k \ell} (\bq-\bz)| \, \rmd \bq \, \rmd \bz \leq \frac{a_2 \alpha_2}{2} + \alpha_2 a_1 a_2 \Vnorm{\ba^{k \ell}}_{L^2(\bbT^d)}\,,
\end{split}
\end{equation*}
which is finite for every $k$ and $\ell$. 
It remains to show that the alternate expression  given in \eqref{alternative-for-C} for $\widetilde{c}^{ijk \ell}$ is equivalent to \eqref{FormulaForC2}. To that end, 
expanding the expression on the right-hand side of \eqref{alternative-for-C}
\begin{equation*}
\begin{split}
\frac{1}{2} & \int_{\bbT^d}  \int_{\bbR^d} \frac{\rho(\bq-\by)}{|\bq-\by|^2}   {\mu_s}(\bq,\by) \Big( (q_i-y_i)(q_j-y_j)(q_k-y_k)(q_{\ell}-y_{\ell}) \Big) \, \rmd \by \, \rmd \bq  \\
	& + \frac{1}{2} \int_{\bbT^d} \int_{\bbR^d}  \frac{\rho(\bq-\by) }{|\bq-\by|^2} {\mu_s}(\bq,\by)(q_k-y_k)(q_{\ell}-y_{\ell}) \big( \ba^{ij}(\bq) - \ba^{ij}(\by) \big) \cdot (\bq-\by)  \, \rmd \by \, \rmd \bq  \\
	&  + \frac{1}{2} \int_{\bbT^d} \int_{\bbR^d} \frac{\rho(\bq-\by) }{|\bq-\by|^2} {\mu_s}(\bq,\by)(q_i-y_i)(q_j-y_j) \big( \ba^{k \ell}(\bq) - \ba^{k \ell}(\by) \big) \cdot (\bq-\by) \, \rmd \by \, \rmd \bq  \\
	& + \frac{1}{2} \int_{\bbT^d} \int_{\bbR^d} \frac{\rho(\bq-\by) }{|\bq-\by|^2} {\mu_s}(\bq,\by)\Big( \big( \ba^{ij}(\bq) - \ba^{ij}(\by) \big) \cdot (\bq-\by) \Big) \\
	&\qquad \qquad \qquad \times \Big( \big( \ba^{k \ell}(\bq) - \ba^{k \ell}(\by) \big) \cdot (\bq-\by) \Big)\, \rmd \by \, \rmd \bq := \mathrm{I} + \mathrm{II} + \mathrm{III} + \mathrm{IV}\,.
\end{split}
\end{equation*}
Note that by \eqref{eq:bOddFxn} for any $i$, $j$, $k$ and $\ell$
\begin{multline}\label{eq:ElasticityTensorProof1}
\frac{1}{2} \int_{\bbT^d} \int_{\bbR^d} \frac{\rho(\bq-\by)}{|\bq-\by|^2} {\mu_s}(\bq,\by) (q_i-y_i)(q_j-y_j) \big( \ba^{k \ell}(\bq)  \cdot (\bq-\by) \big) \, \rmd \by \, \rmd \bq \\
= - \frac{1}{2} \int_{\bbT^d} \int_{\bbR^d} \frac{\rho(\bq-\by) }{|\bq-\by|^2}{\mu_s}(\bq,\by) (q_i-y_i)(q_j-y_j) \big( \ba^{k \ell}(\by) \cdot (\bq-\by) \big) \, \rmd \by \, \rmd \bq\,.
\end{multline}
Therefore,
\begin{equation*}
\begin{split}
\mathrm{III} = - \int_{\bbT^d} \int_{\bbR^d} \frac{\rho(\bq-\by) }{|\bq-\by|^2} {\mu_s}(\bq,\by)(q_i-y_i)(q_j-y_j) \big( \ba^{k \ell}(\by) \cdot (\bq-\by) \big) \, \rmd \by \, \rmd \bq\,,
\end{split}
\end{equation*}
and so we see by \eqref{FormulaForC2} that $\widetilde{c}_{ijk \ell} = \mathrm{I} + \mathrm{III}$. We now show that $\mathrm{IV} = - \mathrm{II}$. Expanding $\mathrm{IV}$,
\begin{equation*}
\begin{split}
\mathrm{IV} &= \frac{1}{2} \int_{\bbT^d} \int_{\bbR^d} \hspace{-1pt} \frac{\rho(\bq-\by) }{|\bq-\by|^2} {\mu_s}(\bq,\by)\Big( \big( \ba^{k \ell}(\bq) - \ba^{k \ell}(\by) \big) \hspace{-1pt} \cdot \hspace{-1pt} (\bq-\by) \Big)  \big( \ba^{ij}(\bq)  \hspace{-1pt} \cdot \hspace{-1pt} (\bq-\by) \big) \, \rmd \by \, \rmd \bq \\
	& \,\, - \frac{1}{2} \int_{\bbT^d} \int_{\bbR^d} \hspace{-1pt} \frac{\rho(\bq-\by) }{|\bq-\by|^2} {\mu_s}(\bq,\by)\Big( \big( \ba^{k \ell}(\bq) - \ba^{k \ell}(\by) \big) \hspace{-1pt} \cdot \hspace{-1pt} (\bq-\by) \Big)  \big( \ba^{ij}(\by) \hspace{-1pt} \cdot \hspace{-1pt} (\bq-\by) \big) \, \rmd \by \, \rmd \bq \\
	&:= (i) + (ii)\,.
\end{split}
\end{equation*}
Since $\ba^{k \ell}$ solves \eqref{eq:CellProblem},
\begin{equation}\label{eq:ElasticityTensorProof2}
\begin{split}
(i) &= \frac{1}{2} \int_{\bbT^d} \ba^{ij}(\bq) \hspace{-1pt} \cdot \hspace{-1pt} \bigg[ \hspace{-1pt} \int_{\bbR^d} \hspace{-1pt} \frac{\rho(\bq-\by) }{|\bq-\by|^2} {\mu_s}(\bq, \by)\Big( \big( \ba^{k \ell}(\bq) - \ba^{k \ell}(\by) \big) \hspace{-1pt} \cdot \hspace{-1pt} (\bq-\by) \Big) (\bq-\by) \, \rmd \by \hspace{-1pt} \bigg] \rmd \bq \\
	&= - \frac{1}{2} \int_{\bbT^d} \ba^{ij}(\bq) \cdot \left[ \int_{\bbR^d} \frac{\rho(\bq-\by) }{|\bq-\by|^2} {\mu_s}(\bq, \by)\Big( (q_k-y_k)(q_{\ell}-y_{\ell}) \Big) (\bq-\by) \, \rmd \by \right] \, \rmd \bq \\
	&= - \frac{1}{2} \int_{\bbT^d}  \int_{\bbR^d} \frac{\rho(\bq-\by) }{|\bq-\by|^2}{\mu_s}(\bq, \by) (q_k-y_k)(q_{\ell}-y_{\ell}) \big( \ba^{ij}(\bq) \cdot (\bq-\by) \big) \, \rmd \by  \, \rmd \bq\,. \\
\end{split}
\end{equation}
Similarly to $(i)$ but additionally using \eqref{eq:aEvenFxn} and \eqref{eq:bOddFxn},
{\small
\begin{equation}\label{eq:ElasticityTensorProof3}
\begin{split}
(ii) &= - \frac{1}{2} \int_{\bbT^d} \int_{\bbR^d} \frac{\rho(\bq-\by) }{|\bq-\by|^2} {\mu_s}(\bq, \by)\Big( \big( \ba^{k \ell}(\bq) - \ba^{k \ell}(\by) \big) \cdot (\bq-\by) \Big)  \Big( \big( \ba^{ij}(\by) \cdot (\bq-\by) \Big) \, \rmd \by \, \rmd \bq \\
	&\EquationReference{\eqref{eq:aEvenFxn}}{=} - \frac{1}{2} \int_{\bbT^d} \int_{\bbR^d} \frac{\rho(\bq-\by) }{|\bq-\by|^2} {\mu_s}(\bq, \by)\Big( \big( \ba^{k \ell}(\by) - \ba^{k \ell}(\bq) \big) \cdot (\bq-\by) \Big)  \Big( \big( \ba^{ij}(\bq) \cdot (\bq-\by) \Big) \, \rmd \by \, \rmd \bq \\
	&= - \frac{1}{2} \int_{\bbT^d}  \ba^{ij}(\bq) \cdot \left( \int_{\bbR^d} \frac{\rho(\bq-\by) }{|\bq-\by|^2} {\mu_s}(\bq, \by)\Big( \big( \ba^{k \ell}(\by) - \ba^{k \ell}(\bq) \big) \cdot (\bq-\by) \Big) (\bq-\by) \, \rmd \by \right) \, \rmd \bq \\
	&\EquationReference{\eqref{eq:CellProblem}}{=} - \frac{1}{2} \int_{\bbT^d} \ba^{ij}(\bq) \cdot \left( \int_{\bbR^d} \frac{\rho(\bq-\by) }{|\bq-\by|^2} {\mu_s}(\bq, \by)\Big( (q_k-y_k)(q_{\ell}-y_{\ell}) \Big) (\bq-\by) \, \rmd \by \right) \, \rmd \bq \\
	&= - \frac{1}{2} \int_{\bbT^d}  \int_{\bbR^d} \frac{\rho(\bq-\by) }{|\bq-\by|^2} {\mu_s}(\bq, \by)(q_k-y_k)(q_{\ell}-y_{\ell}) \big( \ba^{ij}(\bq) \cdot (\bq-\by) \big) \, \rmd \by  \, \rmd \bq \\
	&\EquationReference{\eqref{eq:bOddFxn}}{=}  \frac{1}{2} \int_{\bbT^d}  \int_{\bbR^d} \frac{\rho(\bq-\by) }{|\bq-\by|^2} {\mu_s}(\bq, \by)(q_k-y_k)(q_{\ell}-y_{\ell}) \big( \ba^{ij}(\by) \cdot (\bq-\by) \big) \, \rmd \by  \, \rmd \bq\,. \\
\end{split}
\end{equation}
}
Thus using \eqref{eq:ElasticityTensorProof2} and \eqref{eq:ElasticityTensorProof3}
\begin{equation*}
\mathrm{IV} = (i) + (ii) = - \mathrm{II}\,.
\end{equation*}
The proof is complete.
\end{proof}

\appendix

\section{The elasticity tensor in the absence of homogeneities}

\begin{proof}[Proof of \eqref{eq:LameTensor}]
If we assume that $\mu_s \equiv 1$ and $\rho(\bz) := \widetilde{\rho}(|\bz|)$, then the integrand defining $\bh^{k\ell}$ in \eqref{eq:CellProbData} is an odd function for every $k$ and $\ell$, and thus $\bh^{k \ell} = {\bf 0}$ for all $k$, $\ell$. Thus $(\bbK-\bbG)\ba^{k \ell}=\bh^{k \ell} = { \bf 0 }$ by \eqref{eq:CellProblem-operatorform}. Therefore, it follows from Proposition \ref{prop:KerOfK-G} and \eqref{eq:UniquenessCondFora} that
%since the kernel of $\bbK - \bbG$ consists of constant functions, and since we enforce the condition that $\int_{\bbT^d} \ba^{k \ell}(\bq) \, \rmd \bq = { \bf 0 }$, it follows that 
$\ba^{k \ell} = {\bf 0}$ for every $k$ and $\ell$.
Then using the alternate expression for $\widetilde{c}^{ijk\ell}$ in \eqref{alternative-for-C} and changing coordinates,
\begin{equation*}
\begin{split}
\widetilde{c}^{ijk \ell} &= \frac{1}{2} \int_{\bbT^d} \int_{\bbR^d} \frac{\rho(\bq-\by)}{|\bq-\by|^2} (q_i-y_i)(q_j-y_j)  (q_k-y_k)(q_{\ell}-y_{\ell}) \, \rmd \by \, \rmd \bq \\
	&= \frac{1}{2} \int_{\bbR^d} \frac{\rho(\bz)}{|\bz|^2} z_i z_j z_k z_{\ell} \, \rmd \bz\,.
\end{split}
\end{equation*}
Writing in polar coordinates and using the definition of $a_2$,
\begin{equation*}
\widetilde{c}^{ijk \ell} = \frac{1}{2} \int_0^{\infty} \widetilde{\rho}(r) r^{d+1} \, \rmd r \intdm{\bbS^{d-1}}{w_i w_j w_k w_{\ell} }{\sigma(\bw)} = \frac{a_2}{2} \fint_{\bbS^{d-1}} w_i w_j w_k w_{\ell} \, \rmd \sigma(\bw)\,.
\end{equation*}
Note that, by using rotations and changes of coordinates, the value of the integral defining $\widetilde{c}^{ijk \ell}$ is $0$ unless any two pairs of indices are equal, or if $i=j=k=\ell$. Therefore,
\begin{equation*}
\widetilde{c}^{ijk \ell} = \frac{a_2}{2} \big( p_{ik} \delta_{ij} \delta_{k \ell} + p_{ij} \delta_{i k } \delta_{j \ell} + p_{ij} \delta_{i \ell} \delta_{j k} \big)\,,
\end{equation*}
where for $i,j \in \{1, \ldots, d \}$
\begin{equation*}
p_{ij} := \fint_{\bbS^{d-1}} w_i^2 w_j^2 \, \rmd \sigma(\bw)\,, \quad i \neq j\,, \qquad \text{ and } \qquad  p_{ii} := \frac{1}{3} \fint_{\bbS^{d-1}} w_i^4 \, \rmd \sigma(\bw)\,.
\end{equation*}
By using the integral values
\begin{equation*}
\fint_{\bbS^{d-1}} w_i^4 \, \rmd \sigma(\bw) = \frac{3}{d(d+2)} \,, \qquad \fint_{\bbS^{d-1}} w_i^2 w_j^2 \, \rmd \sigma(\bw) = \frac{1}{d(d+2)}\,, \, i \neq j\,,
\end{equation*}
(see \cite{mengesha2015VariationalLimit}) we see that $\widetilde{c}^{ijkl}$ is exactly of the form \eqref{eq:LameTensor}.
\end{proof}

\bibliographystyle{siamplain}
\bibliography{References}
\end{document}